\documentclass[hidelinks]{article}

\usepackage{amssymb}
\usepackage{latexsym}
\usepackage{amsmath}
\usepackage{array}
\usepackage{graphicx}
\graphicspath{Images/}
\usepackage{caption}
\usepackage{subcaption}
\usepackage{algorithmic}
\usepackage{float}
\usepackage{hyperref}
\usepackage{cleveref}
\usepackage{algorithm}
\usepackage{amsthm}

\newtheorem{theorem}{Theorem}
\newtheorem{definition}{Definition}
\newtheorem{lemma}{Lemma}

\title{An upwind Lattice Boltzmann scheme\footnote{Submitted to the editors on November 13, 2021.}}
\date{}
\author{Megala A\footnote{Department of Aerospace Engineering, Indian Institute of Science, Bangalore, India. (email: megalaa@iisc.ac.in, raghu@iisc.ac.in).} \ and \ S.V. Raghurama Rao\footnotemark[2] }
\begin{document}
\maketitle{}

\begin{abstract}
A lattice Boltzmann scheme that is close to pure upwind, low diffusive and entropy satisfying Engquist-Osher scheme has been formulated for hyperbolic scalar conservation laws. A model for source terms, with scalar conservation laws, is introduced to remove the numerical imbalance between convection and source terms.
\end{abstract}



\section{Introduction}
Lattice Boltzmann Method (LBM) has been established as a popular alternative to traditional Computational Fluid Dynamics (CFD) methods in the last three decades. The popularity of LBM is due to its algorithmic simplicity (convection and collision steps), and further its high accuracy due to zero truncation error in the streaming step. For a review of LBM, the reader is referred to Succi \cite{S2018}. Due to the inherent low Mach number limit, the traditional LBM is restricted to simulation of incompressible flows. Development of lattice Boltzmann methods for hyperbolic conservation laws is currently an active area of research. In this study, we develop an LBM which ensures upwinding at the macroscopic level, for simulating hyperbolic scalar conservation laws. CFD algorithms have reached a high level of maturity and mimicking them can be a good strategy to develop an efficient LBM. Engquist-Osher scheme (\cite{EO}) is a well-known macroscopic scheme that is upwinding, entropy satisfying and has low numerical diffusion. Formulation of an LBM that is close to Engquist-Osher scheme can exploit the advantages of both the schemes. Using flux decomposed equilibrium distribution functions introduced by Aregba-Driollet and Natalini \cite{ADN2000} in one of their discrete kinetic schemes, an LBM that is equivalent to Engquist-Osher scheme upto second order in time has been formulated in this paper. \\ 
Discretisation of source terms in the simulation of hyperbolic conservation laws is non-trivial. When stiff source terms are present, without careful discretisation, discontinuities may appear at wrong locations (see LeVeque and Yee \cite{LY1990}). Source term discretisation has been another active area of research in CFD and various strategies are available. In this paper, the formulated lattice Boltzmann (LB) scheme has been extended to hyperbolic conservation laws with stiff source terms by an appropriate model of source terms at the discrete Boltzmann level, removing the spurious numerical convection terms with the help of Chapman-Enskog expansion of lattice Boltzmann equation.

\section{Formulation of Lattice Boltzmann scheme}
In this section, an LBM has been formulated for hyperbolic scalar conservation laws by using the flux decomposed equilibrium distribution functions introduced by Aregba-Driollet and Natalini \cite{ADN2000}. Chapman-Enskog analysis has been done to obtain modified PDE of the scheme, and an equivalence has been established between new LBM and Engquist-Osher scheme upto second order in time. The total variation boundedness property of the LBM has also been established for one dimensional case. 
\subsection{Flux Decomposition Method}
This is one of the discrete kinetic schemes proposed by Aregba-Driollet and Natalini \cite{ADN2000}. Consider the Cauchy problem, 
 \begin{equation}
 \label{Scalar cons. law}
 \partial_t u + \sum_{d=1}^D \partial_{x_d} g_d (u) = 0 \ \text{with initial condition (IC) as} \ u(\mathbf{x},0)=u_0(\mathbf{x})
 \end{equation}
$u:\mathbb{R}^D \times [0,T] \rightarrow \mathcal{U} \subseteq \mathbb{R}$ is a weak solution to the hyperbolic scalar conservation law, and the flux functions are $g_d(u):\mathcal{U} \rightarrow \mathbb{R}, \ \forall d$. The problem in \cref{Scalar cons. law} is approximated by 
  \begin{equation}
  \label{DVBE}
  \partial_t f^{\epsilon} + \sum_{d=1}^D \Lambda_d \partial_{x_d} f^{\epsilon} = -\frac{1}{\epsilon} \biggl(f^{\epsilon}-f^{eq}(Pf^{\epsilon})\biggr) \ \text{with IC as} \ f^{\epsilon}(\mathbf{x},0)=f_0^{\epsilon}(\mathbf{x})
  \end{equation}
  Here $\epsilon$ is a positive number, $\Lambda_d$ are real diagonal $N \times N$ matrices, $P$ is a real constant coefficients $1 \times N$ matrix, $f^{\epsilon}:\mathbb{R}^D \times [0,T] \rightarrow \mathbb{R}^N$ and $f^{eq}: \mathcal{U} \rightarrow \mathbb{R}^N$. Here $N$ is the number of discrete velocities. $f^{eq}$ satisfies the following relations for all $u \in \mathcal{U}$,
  \begin{equation}
  \label{Moments}
  Pf^{eq}(u)=u \\ \ ; \ P\Lambda_df^{eq}(u)=g_d(u) \  \quad\text{for} \ d \in \{ 1,2,..,D\}
  \end{equation}
  It can be seen that, if $f^{\epsilon}$ converges in some strong topology to a limit $f$ and if $Pf_0^{\epsilon}$ converges to $u_0$, then $Pf$ is a solution of problem \cref{Scalar cons. law}. The approximation in \cref{DVBE} represents a discrete velocity Boltzmann equation with BGK model for \cref{Scalar cons. law}. Numerical schemes for \cref{DVBE} in the limit $\epsilon = 0$ are numerical schemes for \cref{Scalar cons. law}, and these are known as discrete kinetic schemes. \\
  In order to construct the system in \cref{DVBE}, $P, \Lambda_d$ and $f^{eq}$ are required. With $P=\begin{bmatrix} 1 & 1 & . &. &.& 1\end{bmatrix}_{1 \times N}$ and $\Lambda_d=diag\left(v_1^{(d)},v_2^{(d)},...,v_N^{(d)}\right)$ where $v_n^{(d)} \in \mathbb{R}$, the system in \cref{DVBE} can be written as: 
  \begin{equation}
  \label{DVBE_indicial}
  \partial_t f_n^{\epsilon} + \sum_{d=1}^D v_n^{(d)} \partial_{x_d} f_n^{\epsilon} = -\frac{1}{\epsilon} \biggl(f_n^{\epsilon}-f_{n}^{eq}(Pf_n^{\epsilon})\biggr)\ \text{for} \ 1 \leq n \leq N
  \end{equation}
  We have $u^{\epsilon}=\sum_{n=1}^N f_n^{\epsilon} \ \text{and} \ f_n^{\epsilon}(\mathbf{x},t) \in \mathbb{R}$. Taking $N=2D+1$, setting
  \begin{align}
  v_n^{(d)}&= \lambda_d \delta_{nd} ,\quad\text{if} \ n \in \{ 1,2,..,D\}    \label{Discrete velocities_a} \\v_n^{(d)}&= 0  ,\quad\text{if}  \ n=D+1 \label{Discrete velocities_b} \\ v_n^{(d)}&=-\lambda_{n-(D+1)}\delta_{d,n-(D+1)} ,\quad\text{if} \ n \in \{ D+2, .., 2D+1\}  \label{Discrete velocities_c}
  \end{align}
  for some $\lambda_d > 0$, and decomposing the wave speeds $a_d(u)=\partial_u g_d(u)$ into positive and negative parts such that $a_d(u)=a_d(u)^+ - a_d(u)^-$, where
  \begin{equation}
  \label{p wavespeed}
  a_d(u)^+ = \left\{ \begin{matrix}a_d(u) & \mbox{if} & a_d(u) \geq 0 \\ \\ 0 & \mbox{if} & a_d(u)<0 \end{matrix} \right. 
  \end{equation}
  \begin{center} and \end{center}
   \begin{equation}
   \label{n wavespeed}
  a_d(u)^- = \left\{ \begin{matrix}0 & \mbox{if} & a_d(u) \geq 0 \\ \\ -a_d(u) & \mbox{if} & a_d(u)<0 \end{matrix} \right.
  \end{equation}
  the corresponding positive and negative fluxes will become, 
  \begin{equation}
  \label{pm fluxes}
  g_d(u)^{\pm} = \int_0^{u} a_d(u)^{\pm} du 
  \end{equation}
  The associated equilibrium distribution functions are defined as, 
  \begin{align}
f_n^{eq}(u)&= \frac{g_n(u)^+}{\lambda_d} \ ,\quad\text{if} \ n \in \{1,2,...,D \}  \label{Flux decomposed equilibrium distribution functions_a} \\ f_{D+1}^{eq}(u)&=u-\left(\sum_{d=1}^D \frac{g_d(u)^+ + g_d(u)^-}{\lambda_d}\right)  \label{Flux decomposed equilibrium distribution functions_b} \\ f_n^{eq}(u)&=\frac{g_{n-(D+1)}(u)^-}{\lambda_{n-(D+1)}} \ ,\quad\text{if} \ n \in \{D+2,...,2D+1 \}  \label{Flux decomposed equilibrium distribution functions_c}
  \end{align}
 The flux decomposed equilibrium distribution functions in \cref{Flux decomposed equilibrium distribution functions_a,Flux decomposed equilibrium distribution functions_b,Flux decomposed equilibrium distribution functions_c} along with $P$ and $\Lambda_d$ satisfy the moments in \cref{Moments}. In other words, the moments $\sum_{n=1}^N f_n^{eq}(u) = u$ and $\sum_{n=1}^N v_n^{(d)}f_n^{eq}(u)=g_d(u)^+ -  g_d(u)^- = g_d(u)$ are satisfied. In this paper, an LB scheme will be formulated using \cref{Flux decomposed equilibrium distribution functions_a,Flux decomposed equilibrium distribution functions_b,Flux decomposed equilibrium distribution functions_c} and \cref{Discrete velocities_a,Discrete velocities_b,Discrete velocities_c}, and its equivalence with Engquist-Osher scheme upto second order in time will be established. 

 \subsection{Lattice Boltzmann scheme}
Consider the discrete velocity Boltzmann equation in \cref{DVBE_indicial} for any $1 \leq n \leq N$ in the limit $f_n^{\epsilon} \rightarrow f_n$. 
\begin{equation*}
\partial_t f_n + \sum_{d=1}^D v_n^{(d)} \partial_{x_d} f_n= -\frac{1}{\epsilon} \biggl(f_n-f_{n}^{eq}(Pf_n)\biggr)
\end{equation*}
Here $\epsilon$ is the relaxation time, and as $\epsilon \rightarrow 0$, $f_n^{\epsilon} \rightarrow f_n$ in \cref{DVBE_indicial}. Hence $\epsilon \sim O(\mbox{Kn})$, where Kn is the Knudsen number. This equation becomes an ODE along certain characteristic curves, and thereby we have the ODE
\begin{equation}
\frac{d}{dt} f_n = - \frac{1}{\epsilon} \biggl(f_n - f_n^{eq}(Pf_n)\biggr)  \quad\text{along} \quad  \frac{d}{dt} x_d = v_n^{(d)} 
\end{equation}
Forward Euler discretisation of the ODE with time step $\Delta t$ while using $Pf_n=u$ and $\omega=\frac{\Delta t}{\epsilon}$ gives the Lattice Boltzmann equation (LBE), 
\begin{equation}
\label{LBE2}
f_n(\mathbf{x}+\mathbf{v_n}\Delta t, t+\Delta t)=(1-\omega)f_n(\mathbf{x}, t) + \omega f_n^{eq} (u(\mathbf{x}, t))
\end{equation}
The above LBE can be solved by splitting into collision and streaming steps as,
\begin{align}
\mbox{Collision:} \quad f_n^{*}=(1-\omega)f_n(\mathbf{x}, t) + \omega f_n^{eq} (u(\mathbf{x}, t)) \label{collision} \\ \mbox{Streaming:} \quad \quad \quad \quad \quad f_n(\mathbf{x}+\mathbf{v_n}\Delta t, t+\Delta t)=f_n^{*} \label{streaming}
\end{align}
Exact streaming is guaranteed when $\mathbf{v}_n$ is chosen such that \cref{Discrete velocities_a,Discrete velocities_b,Discrete velocities_c} hold and $\lambda_d=\frac{\Delta x_d}{\Delta t}$ for a structured mesh size of $\Delta x_d$ in the direction $d$. For uniform lattice, we have $\lambda_d=\lambda, \ \forall d$. The definition of $f_n^{eq}$ in the LBE \cref{LBE2} follows \cref{Flux decomposed equilibrium distribution functions_a,Flux decomposed equilibrium distribution functions_b,Flux decomposed equilibrium distribution functions_c}. \Cref{alg:LB} can be followed to solve \cref{LBE2} by using the conserved moment $\sum_{n=1}^N f_n = \sum_{n=1}^N f_n^{eq}=u$ to evaluate $u$, and the form of $g_d(u)$ to evaluate $g_d$ at the end of each time step.

\begin{algorithm}
\caption{LB algorithm}\label{alg:LB}
\begin{algorithmic}[1]
\STATE{Evaluate $f_n^{eq}(\mathbf{x},0)$ from $u_0(\mathbf{x})$ using \cref{Flux decomposed equilibrium distribution functions_a,Flux decomposed equilibrium distribution functions_b,Flux decomposed equilibrium distribution functions_c}, for $n \in \{1,2,..,N\}$ and $\forall \mathbf{x}$ in the lattice.}
\STATE{Initialise $f_n(\mathbf{x},0)=f_n^{eq}(\mathbf{x},0)$ for $n \in \{1,2,..,N\}$ and $\forall \mathbf{x}$ in the lattice, and take $t=0$.}
\WHILE{$T-t>10^{-8}$}
\STATE{Carry out Collision step using \cref{collision} for an appropriate choice of $\omega$ dictated by the numerical diffusion co-efficient obtained through  Chapman-Enskog analysis in \cref{sec:CE analysis}.} 
\STATE{Carry out Streaming step using \cref{streaming} for appropriate choice of discrete velocities $\mathbf{v_n}$ such that \cref{Discrete velocities_a,Discrete velocities_b,Discrete velocities_c} hold. For fixed $\lambda_d=\lambda>0$ where $d \in \{ 1,2,..,D\}$ and uniform lattice spacing, $\Delta t = \frac{\Delta x_d}{\lambda_d}=\frac{\Delta x}{\lambda}$ is determined uniquely for any $d$. The choice of $\lambda$ is dictated by the numerical diffusion co-efficient obtained through  Chapman-Enskog analysis in \cref{sec:CE analysis}.}
\STATE{Use appropriate boundary conditions for $f_n$.}
\STATE{Evaluate $u(\mathbf{x},t+\Delta t)$ using $\sum_{n=1}^{N} f_n=u$.}
\STATE{Evaluate $g_d(u(\mathbf{x},t+\Delta t))$ using $u(\mathbf{x},t+\Delta t)$.}
\STATE{Evaluate $f_n^{eq}(\mathbf{x},t+\Delta t)$ using \cref{Flux decomposed equilibrium distribution functions_a,Flux decomposed equilibrium distribution functions_b,Flux decomposed equilibrium distribution functions_c} for $n \in \{1,2,..,N\}$.}\STATE{Update $t=t+\Delta t$}
\ENDWHILE
\RETURN $u$
\end{algorithmic}
\end{algorithm}

\subsection{Chapman-Enskog analysis of the LB scheme}
\label{sec:CE analysis}
Taylor expanding the LBE in \cref{LBE2} and simplifying, we get
\begin{multline}
\label{TE2}
\left( \partial_t + \sum_{d=1}^D v_n^{(d)} \partial_{x_d}\right) f_n =  - \frac{1}{\epsilon} \left(f_n-f_n^{eq}\right) \\ + \frac{\Delta t}{2\epsilon} \left( \partial_t + \sum_{d=1}^D v_n^{(d)} \partial_{x_d}\right) (f_n-f_n^{eq}) + O(\Delta t^2)
\end{multline}
Consider the perturbation expansion of $f_n$,
\begin{equation}
\label{PE}
f_n=f_n^{eq} + f_n^{neq} \ \mbox{where} \ f_n^{neq}=\xi f_n^{(1)} + \xi^2 f_n^{(2)} + . . .
\end{equation}
Here $\xi \sim O(\mbox{Kn)}$. Since BGK collision operator is invariant under the conserved moment, we have $\sum_{n=1}^N f_n = \sum_{n=1}^N f_n^{eq} = u$, and the moment of non-equilibrium distribution function becomes $\sum_{n=1}^N \left(\xi f_n^{(1)} + \xi^2 f_n^{(2)}+ . . \right)=0$. Each term corresponding to different order of $\xi$ in this moment expression must individually be zero. Hence, we have $\sum_{n=1}^N f_n^{(i)} = 0 \ \mbox{for} \ i \in \mathbb{N}$. Multiple scale expansion of derivatives of $f_n$ gives $\partial_t f_n = \left( \xi \partial_t^{(1)} + \xi^2 \partial_t^{(2)} + ... \right) f_n \ \text{and} \ v_n^{(d)} \partial_{x_d} f_n = \xi v_n^{(d)} \partial_{x_d}^{(1)} f_n$. \\
Using perturbation expansion of $f_n$ and multiple scale expansion of derivatives of $f_n$ in \cref{TE2} and separating out $O(\xi)$ and $O(\xi^2)$ terms,
 \begin{align}
 O(\xi)&: \left( \partial_t^{(1)} + \sum_{d=1}^D v_n^{(d)} \partial_{x_d}^{(1)}\right) f_n^{eq} = - \frac{1}{\epsilon} f_n^{(1)}  \label{O(xi)} \\
 O(\xi^2)&: \ \partial_t^{(2)} f_n^{eq} + \left( 1- \frac{\Delta t}{2\epsilon}\right)\left( \partial_t^{(1)} + \sum_{d=1}^D v_n^{(d)} \partial_{x_d}^{(1)}\right) f_n^{(1)} =  - \frac{1}{\epsilon} f_n^{(2)}  \label{O(xi^2)}
 \end{align}
Zeroth moment of $O(\xi)$ terms in \cref{O(xi)} and $O(\xi^2)$ terms in \cref{O(xi^2)} respectively give,
 \begin{align}
 \partial_t^{(1)} u + \sum_{d=1}^D \partial_{x_d}^{(1)} g_d(u) = 0 \label{0thO(xi)} \\
 \partial_t^{(2)} u + \left( 1- \frac{\Delta t}{2\epsilon}\right) \sum_{d=1}^D \partial_{x_d}^{(1)} \left(\sum_{n=1}^N v_n^{(d)}  f_n^{(1)}\right)=0  \label{0thO(xi^2)}
 \end{align}
From the first moment of $O(\xi)$ terms in \cref{O(xi)}, we get 
\begin{equation}
\label{1stO(xi)}
\sum_{n=1}^N v_n^{(d)} f_n^{(1)} = - \epsilon \left( \partial_u g_d \left( -\sum_{i=1}^D \partial_u g_i \partial_{x_i}^{(1)} u \right) + \sum_{i=1}^D \partial_{x_i}^{(1)} \left( \sum_{n=1}^N v_n^{(d)} v_n^{(i)} f_n^{eq} \right)\right)
\end{equation}
Recombining the zeroth moment equations of $O(\xi)$ in \cref{0thO(xi)} and $O(\xi^2)$ in \cref{0thO(xi^2)}, further reversing the multiple scale expansions, substituting \cref{1stO(xi)} and using $\epsilon \left( 1- \frac{\Delta t}{2\epsilon}\right) = \Delta t \left( \frac{1}{\omega}-\frac{1}{2}\right)$, we get 
\begin{multline}
\label{mPDE2}
\partial_t u + \sum_{d=1}^D \partial_{x_d} g_d(u) = \\ \Delta t \left( \frac{1}{\omega}-\frac{1}{2}\right) \sum_{d=1}^D \partial_{x_d} \left(  \sum_{i=1}^D \partial_{x_i} \left( \sum_{n=1}^N v_n^{(d)} v_n^{(i)} f_n^{eq} \right) - \partial_u g_d \left( \sum_{i=1}^D \partial_u g_i \partial_{x_i} u \right)\right)
\end{multline}
\Cref{mPDE2} is correct upto $O(\Delta t^2)$, because the analysis is carried out on \cref{TE2} which is correct upto $O(\Delta t^2)$. It can be seen from \cref{mPDE2} that the LB scheme solves an $O(\Delta t)$ approximation of the scalar conservation law. However, empirical data (\cref{tab:EOC}, \cref{EOC}) for the numerical solution of one dimensional inviscid Burgers' equation with an initial sine wave profile shows that the order of accuracy can be close to, and sometimes even more than, two. 
\begin{table}[tbhp]
\begin{center}
\begin{tabular}{|m{2cm}|m{2cm}|m{2cm}|m{2cm}|}
\hline
\centering Number of lattice points, N & \centering Lattice spacing, h & \centering $L_2$ norm & EOC using $L_2$ norm \\
\hline
\centering 40 & \centering 0.025 & \centering 0.00979288 &  - \\
\centering 80 & \centering 0.0125 & \centering 0.00327174 &  1.581675 \\
\centering 160 & \centering 0.00625 & \centering 0.000895927 &  1.868605 \\
\centering 320 & \centering 0.003125 & \centering 0.000132689 &  2.755337 \\
\hline
\end{tabular}
\caption{ Empirical data on Experimental Order of Convergence for 1D Inviscid Burgers' equation with initial sine wave profile}
\label{tab:EOC}
\end{center}
\end{table}
\begin{figure}[h]
\centering
\includegraphics[scale=0.6]{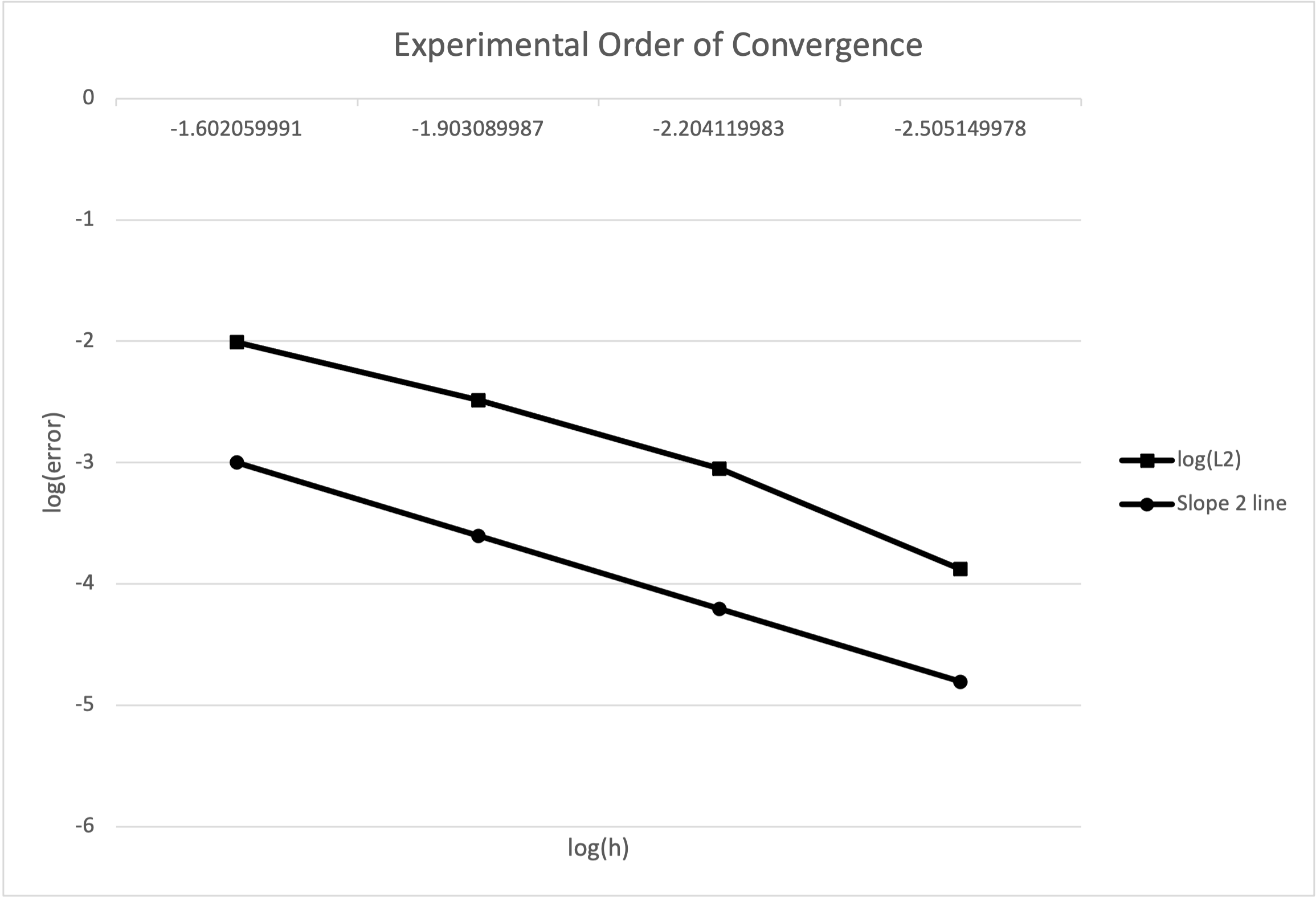}
\caption{ Comparison of EOC using L2 norm with Slope 2 line for 1D inviscid Burgers' equation with initial sine wave profile}
\label{EOC}
\end{figure}\\
For one dimensional scalar conservation laws, the modified PDE in \cref{mPDE2} becomes,
\begin{equation}
\partial_t u +  \partial_{x_1} g_1(u) = \Delta t \left( \frac{1}{\omega}-\frac{1}{2}\right) \partial_{x_1} \biggl(\biggl( \lambda_1-\partial_u |g_1|\biggr)\partial_u |g_1|\partial_{x_1}u \biggr)
\end{equation}
It can be seen that numerical diffusion coefficient is non-negative when $0<\omega<2$ and $\lambda_1 \geq \sup_{u\in \mathcal{U}} \partial_u |g_1|$.   For two dimensional scalar conservation laws on uniform lattice, the modified PDE in \cref{mPDE2} becomes,
\begin{multline}
\partial_t u +  \partial_{x_1} g_1(u) + \partial_{x_2} g_2(u) = \\  \Delta t \left( \frac{1}{\omega}-\frac{1}{2}\right)   \partial_{x_1} \biggl(\biggl(\lambda-\partial_u |g_1|\biggr)\partial_u |g_1|\partial_{x_1}u - \partial_u g_1 \partial_u g_2\partial_{x_2}u\biggr)  \\  + \Delta t \left( \frac{1}{\omega}-\frac{1}{2}\right) \partial_{x_2} \biggl( - \partial_u g_2 \partial_u g_1\partial_{x_1}u + \biggl(\lambda-\partial_u |g_2|\biggr)\partial_u |g_2|\partial_{x_2}u\biggr)
\end{multline}
Numerical diffusion co-efficient is non-negative when both $0<\omega<2$ and positive semi-definiteness of the matrix
\begin{equation}
\begin{bmatrix}
\biggl(\lambda-\partial_u |g_1|\biggr)\partial_u |g_1| & - \partial_u g_1 \partial_u g_2 \\ - \partial_u g_2 \partial_u g_1 & \biggl(\lambda-\partial_u |g_2|\biggr)\partial_u |g_2|
\end{bmatrix}
\end{equation}  
through an appropriate choice of $\lambda$ are ensured.  

\subsection{Equivalence of LB scheme with Engquist-Osher scheme} 
In this subsection, an equivalence between formulated LB scheme and Engquist-Osher scheme has been established upto $O(\Delta t^2)$.
\begin{theorem}\label{thm1}
LB scheme (given by the LBE in \cref{LBE2}) satisfying the definition of conserved moment $\sum_{n=1}^N f_n = \sum_{n=1}^N f_n^{eq} = u$ along with the form of discrete velocities in \cref{Discrete velocities_a,Discrete velocities_b,Discrete velocities_c}, flux decomposed equilibrium distribution functions in \cref{Flux decomposed equilibrium distribution functions_a,Flux decomposed equilibrium distribution functions_b,Flux decomposed equilibrium distribution functions_c} and perturbation expansion in \cref{PE} corresponds to the macroscopic scheme, 
\begin{multline}
u\left(x_1,..,x_d,..x_D, t+\Delta t\right) = u\left(x_1,..,x_d,..x_D, t\right) \\
- \sum_{d=1}^D \frac{\Delta t}{\Delta x_d} \left(g_d^{+}\left(x_1,..,x_d,..,x_D,t\right) - g_d^{+}\left(x_1,..,x_{d-1}, x_d-\Delta x_d, x_{d+1},..,x_D,t\right) \right) \\ 
+ \sum_{d=1}^D \frac{\Delta t}{\Delta x_d} \left(g_d^{-}\left(x_1,..,x_{d-1}, x_d+ \Delta x_d, x_{d+1},..,x_D,t\right) - g_d^{-}\left(x_1,..,x_d,..,x_D,t\right) \right)
\end{multline}
upto $O(\Delta t^2)$.  
\end{theorem} 
\begin{proof}
Let's rewrite the LBE in \cref{LBE2} as,
\begin{multline}
\label{LBE3}
f_n(x_1,..,x_d,..x_D, t+\Delta t) = \\
(1-\omega)f_n\left(x_1-v_n^{(1)}\Delta t,..,x_d-v_n^{(d)}\Delta t,..,x_D-v_n^{(D)}\Delta t, t\right) \\
+ \omega f_n^{eq}\left(u\left(x_1-v_n^{(1)}\Delta t,..,x_d-v_n^{(d)}\Delta t,..,x_D-v_n^{(D)}\Delta t, t\right)\right)
\end{multline}
Substituting the perturbation expansion $f_n=f_n^{eq}+f_n^{neq}$ on RHS of the re-written form of LBE in \cref{LBE3} and summing over $n$, we get
\begin{multline}
\label{closeness1}
\sum_{n=1}^N f_n(x_1,..,x_d,..x_D, t+\Delta t) = \\  \sum_{n=1}^N f_n^{eq}\left(u\left(x_1-v_n^{(1)}\Delta t,..,x_d-v_n^{(d)}\Delta t,..,x_D-v_n^{(D)}\Delta t, t\right)\right) \\  + (1-\omega)\sum_{n=1}^N f_n^{neq}\left(u\left(x_1-v_n^{(1)}\Delta t,..,x_d-v_n^{(d)}\Delta t,..,x_D-v_n^{(D)}\Delta t, t\right)\right)
\end{multline}
From the definition of conserved moment $\sum_{n=1}^N f_n = \sum_{n=1}^N f_n^{eq} = u$, we have
\begin{equation}
\label{M1}
\sum_{n=1}^N f_n(x_1,..,x_d,..x_D, t+\Delta t) = u(x_1,..,x_d,..x_D, t+\Delta t) 
\end{equation}
From the definition of discrete velocities in \cref{Discrete velocities_a,Discrete velocities_b,Discrete velocities_c} and flux decomposed equilibrium distribution functions in \cref{Flux decomposed equilibrium distribution functions_a,Flux decomposed equilibrium distribution functions_b,Flux decomposed equilibrium distribution functions_c}, we get
\begin{multline}
\label{M2}
\sum_{n=1}^N f_n^{eq}\left(u\left(x_1-v_n^{(1)}\Delta t,..,x_d-v_n^{(d)}\Delta t,..,x_D-v_n^{(D)}\Delta t, t\right)\right) = \\  \sum_{d=1}^D \frac{1}{\lambda_d} \left( g_d^{+}(x_1,..,x_{d-1}, x_d-\lambda_d \Delta t, x_{d+1},..,x_D,t)\right) + \\ u(x_1,..,x_d,..x_D, t) - \sum_{d=1}^D \frac{1}{\lambda_d} \left( g_d^{+}(x_1,..,x_d,..,x_D,t) + g_d^{-}(x_1,..,x_d,..,x_D,t)\right)   \\ +\sum_{d=1}^D \frac{1}{\lambda_d} \left( g_d^{-}(x_1,..,x_{d-1}, x_d+\lambda_d \Delta t, x_{d+1},..,x_D,t)\right)
\end{multline}
From \cref{O(xi)} we have, $f_n^{(1)} = -\epsilon \left( \partial_t^{(1)} + \sum_{d=1}^D v_n^{(d)} \partial_{x_d}^{(1)}\right) f_n^{eq} \sim O(\epsilon)$.
Since $f_n^{neq} = \xi f_n^{(1)} + \xi^2 f_n^{(2)}+... \sim O(\xi f_n^{(1)})$, we have $f_n^{neq} \sim O(\xi \epsilon)$. Since $O(\xi) \sim O(\text{Kn}) \sim O(\epsilon)$, we have $f_n^{neq} \sim O(\epsilon^2)$. Since $0<\omega<2$, we can choose $\omega$ between $1 \ \text{and} \ 2$ and thereby we have, $\frac{1}{\omega} \sim O(1)$. Using these, we get
\begin{multline}
\label{M3}
(1-\omega)\sum_{n=1}^N f_n^{neq}\left(u\left(x_1-v_n^{(1)}\Delta t,..,x_d-v_n^{(d)}\Delta t,..,x_D-v_n^{(D)}\Delta t, t\right)\right)  \\ \sim (1-\omega) O(\epsilon^2) \sim (1-\omega)\epsilon^2 O(1) \sim \Delta t^2 \frac{1}{\omega} \left(\frac{1}{\omega}-1\right) O(1) \quad \left(\because \ \epsilon=\frac{\Delta t}{\omega} \right)\\ \sim \Delta t^2 O(1) \sim O(\Delta t^2)
\end{multline}
Using \cref{M1}, \cref{M2}, \cref{M3} and $\lambda_d=\frac{\Delta x_d}{\Delta t}$ in \cref{closeness1}, we get the macroscopic scheme 
\begin{multline}
\label{macro}
u\left(x_1,..,x_d,..x_D, t+\Delta t\right) = u\left(x_1,..,x_d,..x_D, t\right) \\
- \sum_{d=1}^D \frac{\Delta t}{\Delta x_d} \left(g_d^{+}\left(x_1,..,x_d,..,x_D,t\right) - g_d^{+}\left(x_1,..,x_{d-1}, x_d- \Delta x_d, x_{d+1},..,x_D,t\right) \right) \\ 
+ \sum_{d=1}^D \frac{\Delta t}{\Delta x_d} \left(g_d^{-}\left(x_1,..,x_{d-1}, x_d+ \Delta x_d, x_{d+1},..,x_D,t\right) - g_d^{-}\left(x_1,..,x_d,..,x_D,t\right) \right) \\
+ O(\Delta t^2)
\end{multline}
\end{proof}

\subsection{TV boundedness of numerical solution}
Boundedness of total variation of numerical solution from the formulated upwind LB scheme for one dimensional scalar conservation laws is studied here. So, the discrete velocities in \cref{Discrete velocities_a,Discrete velocities_b,Discrete velocities_c} and flux decomposed equilibrium distribution functions in \cref{Flux decomposed equilibrium distribution functions_a,Flux decomposed equilibrium distribution functions_b,Flux decomposed equilibrium distribution functions_c} corresponding to $D=1$, along with the LBE in \cref{LBE2} will be used to understand the property. For convenience, let's drop the dimensional index and denote $x_1 \ \mbox{as} \ x$, $g_1 \ \mbox{as} \ g$ and $\lambda_1 \ \mbox{as} \ \lambda$ since we are dealing with one dimensional scalar conservation laws. Let's also denote $u(x,t) \ \mbox{as} \ u_i^m$ and $f_n(x,t) \ \mbox{as} \ f_{n_i}^m$ on the lattice structure. Subscript $i$ and superscript $m$ index spatial variable $x$ and temporal variable $t$ respectively.  From the LBE in \cref{LBE2}, we have the following update equations for distribution functions on the lattice structure.
\begin{align}
f_{1_i}^{m+1} &= (1-\omega) f_{1_{i-1}}^m + \omega f_{1_{i-1}}^{eq^m} \label{D1Q3LBE_a}\\  
f_{2_i}^{m+1} &= (1-\omega) f_{2_{i}}^m + \omega f_{2_{i}}^{eq^m} \label{D1Q3LBE_b}\\ 
f_{3_i}^{m+1} &= (1-\omega) f_{3_{i+1}}^m + \omega f_{3_{i+1}}^{eq^m} \label{D1Q3LBE_c}
\end{align}
\begin{definition}
The total variation of any variable $q$ defined on a lattice structure indexed by $i$ is given by,
\begin{equation*}
\mathbf{TV}(q) = \sum_i \left| q_{i+1}-q_i\right|
\end{equation*}
\end{definition}
Let's now show that the total variation of numerical solution $u^m$ obtained using the formulated LB scheme on lattice structure indexed by $i$ remains bounded at any time $t=0+m\Delta t$.  
\begin{lemma}
\label{lem1}
If $u_i^m \in \mathcal{U}^m \ \forall \ i : x \pm i\Delta x  \in \Omega $ and $\lambda \geq \sup_{\zeta \ \in \  \mathcal{U}^m} \left| g'(\zeta) \right|$, then 
\begin{align}
\mathbf{TV}(f_{1}^{eq^m}) &\leq  \mathbf{TV}(u^m) \label{Lem1_a}\\ 
\mathbf{TV}(f_{2}^{eq^m}) &\leq 2 \ \mathbf{TV}(u^m) \label{Lem1_b}\\ 
\mathbf{TV}(f_{3}^{eq^m}) &\leq  \mathbf{TV}(u^m) \label{Lem1_c}
\end{align}
using the definition of flux decomposed equilibrium distribution functions in \cref{Flux decomposed equilibrium distribution functions_a,Flux decomposed equilibrium distribution functions_b,Flux decomposed equilibrium distribution functions_c} corresponding to $D=1$.  
\end{lemma}
\begin{proof}
\begin{equation*}
\mathbf{TV}(f_{n}^{eq^m}) = \sum_{i} \left| f_{n_{i+1}}^{eq^m} - f_{n_i}^{eq^m} \right|
\end{equation*}
For $n=1$,
\begin{subequations}
\begin{align*}
\mathbf{TV}(f_{1}^{eq^m}) &=  \sum_{i} \left| f_{1_{i+1}}^{eq^m} - f_{1_i}^{eq^m} \right|  \\ 
&=  \frac{1}{\lambda} \sum_{i} \left| g_{i+1}^{+^m} - g_{i}^{+^m}\right| \\ 
&\leq  \frac{1}{\lambda} \sup\limits_{\zeta \ \in \  \mathcal{U}^m} \left| g'(\zeta) \right| \sum_{i} \left| u_{i+1}^{m} - u_{i}^{m}\right|  \\
&\leq  \sum_{i} \left| u_{i+1}^{m} - u_{i}^{m}\right|  \\
&=  \mathbf{TV}(u^m) 
\end{align*}
\end{subequations} 
Similar algebra for $n=2$ and $n=3$ gives, respectively, $\mathbf{TV}(f_{2}^{eq^m}) \leq 2 \ \mathbf{TV}(u^m)$ and $\mathbf{TV}(f_{3}^{eq^m}) \leq  \mathbf{TV}(u^m)$. 
\end{proof}
In general, we have
\begin{equation}
\label{LEM1}
\mathbf{TV}(f_{n}^{eq^m}) \leq 2 \ \mathbf{TV}(u^m) \ \mbox{for} \ n \in \{1,2,3 \}
\end{equation}

\begin{lemma}
\label{lem2}
For the update formulae of distribution functions given by \cref{D1Q3LBE_a,D1Q3LBE_b,D1Q3LBE_c}, 
\begin{equation}
\label{Lem2}
\mathbf{TV}(f_n^m)\leq \left|1-\omega \right|\mathbf{TV}(f_n^{m-1}) + \omega \ \mathbf{TV}(f_n^{eq^{m-1}}) \ ; \ n \in \{ 1,2,3\}
\end{equation}
\end{lemma}
\begin{proof}
\begin{equation*}
\mathbf{TV}(f_n^m) = \sum_i \left| f_{n_{i+1}}^m - f_{n_{i}}^m\right|
\end{equation*}
Using \cref{D1Q3LBE_a,D1Q3LBE_b,D1Q3LBE_c}, we have
\begin{align*}
\mathbf{TV}(f_1^m) =  \sum_i \left| (1-\omega)f_{1_{i}}^{m-1} + \omega f_{1_{i}}^{eq^{m-1}} - (1-\omega) f_{1_{i-1}}^{m-1} - \omega f_{1_{i-1}}^{eq^{m-1}} \right| \\ 
\mathbf{TV}(f_2^m)  =  \sum_i \left| (1-\omega)f_{2_{i+1}}^{m-1} + \omega f_{2_{i+1}}^{eq^{m-1}} - (1-\omega) f_{2_{i}}^{m-1} - \omega f_{2_{i}}^{eq^{m-1}} \right| \\ 
\mathbf{TV}(f_3^m)  =  \sum_i \left| (1-\omega)f_{3_{i+2}}^{m-1} + \omega f_{3_{i+2}}^{eq^{m-1}} - (1-\omega) f_{3_{i+1}}^{m-1} - \omega f_{3_{i+1}}^{eq^{m-1}} \right|
\end{align*} 
It can be inferred from these expressions that for $n \in \{ 1,2,3\}$,
\begin{equation*}
\mathbf{TV}(f_n^m) \leq |1-\omega| \left( \sum_i \left|f_{n_{i+1}}^{m-1} - f_{n_{i}}^{m-1}\right| \right) + \omega \left( \sum_i \left|f_{n_{i+1}}^{eq^{m-1}} -  f_{n_{i}}^{eq^{m-1}}\right| \right) 
\end{equation*}
Hence,
\begin{equation*}
\mathbf{TV}(f_n^m)\leq \left|1-\omega \right|\mathbf{TV}(f_n^{m-1}) + \omega \ \mathbf{TV}(f_n^{eq^{m-1}}) 
\end{equation*}
\end{proof}

\begin{theorem}
If the conditions in Lemma \ref{lem1} and Lemma \ref{lem2} are satisfied along with $0<\omega<2$, then 
\begin{equation}
\label{Thm2}
\mathbf{TV}(u^{m}) \leq \mathcal{M}^m \ \mathbf{TV}(u^0) 
\end{equation} 
for some positive $\mathcal{M}^m$ dependent on the time level m, provided, at any lattice point, $\sum_{n=1}^N f_n^m = \sum_{n=1}^N f_n^{eq^m} = u^m$ and $f_n^0 = f_n^{eq^0}$ are satisfied . 
\end{theorem}
\begin{proof}
Let's assume that the conditions in Lemma \ref{lem1} and Lemma \ref{lem2} are satisfied, so that Lemma \ref{lem1} and Lemma \ref{lem2} hold true. Repeated substitution of \cref{LEM1} and \cref{Lem2} in \cref{Lem2} gives,
\begin{multline}
\label{TV1}
\mathbf{TV}(f_n^m) \leq  \left|1-\omega\right|^m \mathbf{TV}(f_n^0) \ + \\ 
 2\omega \left( \left|1-\omega\right|^{m-1} \mathbf{TV}(u^0) + ...+  \left|1-\omega\right|^{0} \mathbf{TV}(u^{m-1}) \right)
\end{multline}
Since $f_{n_i}^0 = f_{n_i}^{eq^0} \ \forall i$, we have $\mathbf{TV}(f_n^0)=\mathbf{TV}(f_n^{eq^0}) \leq 2 \ \mathbf{TV}(u^0)$. Using this in \cref{TV1}, we get
\begin{multline}
\label{TV3}
\mathbf{TV}(f_n^m)  \leq  2\ \left|1-\omega\right|^m \mathbf{TV}(u^0) \ + \\ 
 2\omega \left( \left|1-\omega\right|^{m-1} \mathbf{TV}(u^0) + ...+  \left|1-\omega\right|^{0} \mathbf{TV}(u^{m-1}) \right)
\end{multline} 
Let us now evaluate $\mathbf{TV}(u^{m+1})$. 
\begin{subequations}
\begin{align*}
\mathbf{TV}(u^{m+1}) &=  \sum_i \left| u_{i+1}^{m+1}-u_i^{m+1}\right| \\ 
 &=  \sum_i \left| \sum_{n=1}^3 f_{n_{i+1}}^{m+1}- \sum_{n=1}^3f_{n_i}^{m+1}\right|
 \end{align*}
\end{subequations}
Using \cref{D1Q3LBE_a,D1Q3LBE_b,D1Q3LBE_c} in the above equation and simplifying, the following inequality is obtained. 
\begin{multline*}
\mathbf{TV}(u^{m+1})  \leq  |1-\omega| \left( \sum_i  \left| f_{1_i}^m-f_{1_{i-1}}^m \right|+ \sum_i  \left| f_{2_{i+1}}^m-f_{2_{i}}^m\right| + \sum_i \left|f_{3_{i+2}}^m-f_{3_{i+1}}^m\right| \right) \\ 
 + \omega \left( \sum_i \left| f_{1_i}^{eq^m}-f_{1_{i-1}}^{eq^m} \right| + \sum_i \left| f_{2_{i+1}}^{eq^m}-f_{2_{i}}^{eq^m} \right| + \sum_i  \left| f_{3_{i+2}}^{eq^m}-f_{3_{i+1}}^{eq^m}\right| \right)
\end{multline*}
Hence,
\begin{multline}
\label{TV4}
\mathbf{TV}(u^{m+1})  \leq   |1-\omega| \left( \mathbf{TV}(f_1^m) + \mathbf{TV}(f_2^m) + \mathbf{TV}(f_3^m)\right) \\ 
 + \omega \left( \mathbf{TV}(f_1^{eq^m}) + \mathbf{TV}(f_2^{eq^m}) + \mathbf{TV}(f_3^{eq^m})\right)
\end{multline}
Substituting \cref{TV3} and \cref{Lem1_a,Lem1_b,Lem1_c} in \cref{TV4}, we have
\begin{multline}
\mathbf{TV}(u^{m+1})  \leq 4\omega \ \mathbf{TV}(u^m) \ +\\ 6\omega \left( |1-\omega| \mathbf{TV} (u^{m-1})+...+|1-\omega|^m \mathbf{TV} (u^0)\right) \\  + 6|1-\omega|^{m+1}\mathbf{TV}(u^0)
\end{multline}
Doing repeated substitution of $\mathbf{TV}(u^k) \ \mbox{for} \ k \in \{1,..,m \}$ in the above inequality until all the terms become some multiple of $\mathbf{TV}(u^0)$, we get
\begin{equation}
\mathbf{TV}(u^{m+1}) \leq \left( \sum_{p=0}^{m+1} M_p \ \omega^{m+1-p} |1-\omega|^{p} \right) \mathbf{TV}(u^0)
\end{equation} 
where $M_p \ \mbox{for} \ p \in \{0,1,..,m+1 \}$ are finite positive constants. As $0<\omega<2$, we have $0<|1-\omega|<1$. Since $M_p, \ \omega^{m+1-p}$ and $|1-\omega|^p$ are finite, $\sum_{p=0}^{m+1} M_p \ \omega^{m+1-p} |1-\omega|^{p} = \mathcal{M}^{m+1}$ is a finite positive constant. So, we have
\begin{equation}
\mathbf{TV}(u^{m+1}) \leq \mathcal{M}^{m+1} \mathbf{TV}(u^0)
\end{equation}
\end{proof}
If $\mathbf{TV}(u^0)<\infty$, then from \cref{Thm2}, we have 
\begin{equation}
\mathbf{TV}(u^{m})<\infty
\end{equation} 
since $\mathcal{M}^m$ is a finite positive constant. Hence, total variation of solution $u^m$ from the formulated LB scheme for one dimensional scalar conservation laws is bounded. 

\subsection{Numerical results}
In this subsection, numerical results for some test problems governed by one and two dimensional scalar conservation laws are provided.  
\subsubsection{One dimensional linear convection equation}
This test case is from Chen and Shu \cite{CS2017}. Consider the one dimensional linear convection equation $\partial_t u + \partial_x g(u) = 0 \ \mbox{with} \ g(u)=u$ in the domain $[0,2\pi]$. Initial condition is $u(x,0)=sin^4(x)$ and boundary conditions are periodic. Numerical solution at $T=2\pi$ obtained by splitting the domain into 41 evenly spaced lattice points is shown in \cref{4}.
\begin{figure}[h!] 
\centering
\includegraphics[width=0.5\textwidth]{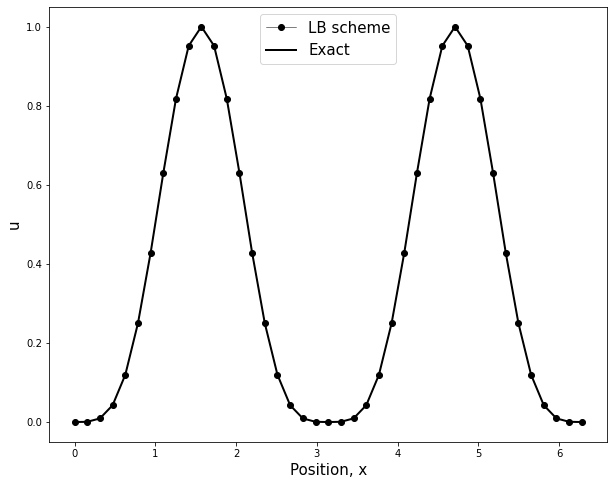}
\caption{ One dimensional linear convection equation}
\label{4}
\end{figure}
\subsubsection{Two dimensional linear scalar conservation law}
Consider the following linear scalar conservation law, $\partial_t u+ \partial_{x_1} g_1(u) + \partial_{x_2} g_2(u) = 0 \ \text{with} \ g_1(u)=au, \ g_2(u)=bu$. This equation is solved for various boundary and initial conditions using the formulated LB scheme. \\
\emph{Discontinuities at various angles:}
This test case is from Spekreijse \cite{Spe1987}. The domain is $[0,1] \times [0,1]$ and $a=cos \ \theta$, $b=sin \ \theta$ with $\theta \in \left(0, \frac{\pi}{2}\right)$. Boundary conditions of the steady-state problem are $u(0,x_2)=1 \ \text{for} \ 0<x_2<1$ and $u(x_1,0)=0 \ \text{for} \ 0<x_1<1$. Exact solution of the steady-state problem is, $u(x_1,x_2)=1 \ \text{for} \ bx_1-ax_2<0$ and $u(x_2,x_2)=0 \ \text{for} \ bx_1-ax_2>0$. For numerical solution using the formulated LB scheme, the domain is split up into $65 \times 65$ evenly spaced lattice points. The numerical solutions plotted for $\theta = 15, \ 30, \ 45, \ 60 \ \mbox{and} \ 75$ are shown in \cref{21a,21b,21c,21d,21e}. It can be seen that the upwind LB scheme captures discontinuities that are not aligned with the lattice coordinate directions reasonably well. 
 \begin{figure}
\centering
\begin{subfigure}[b]{0.31\textwidth}
\centering
\includegraphics[width=\textwidth]{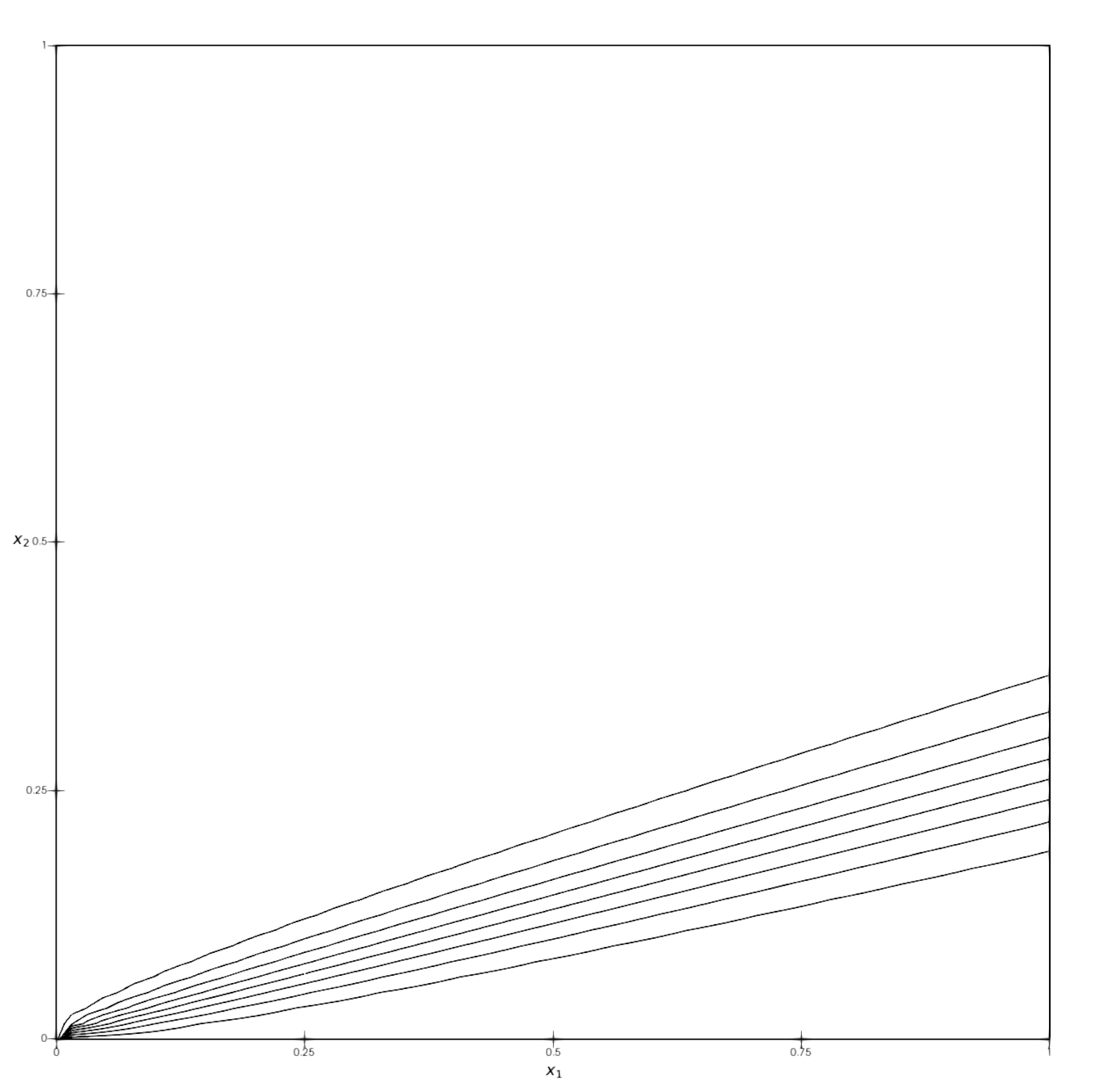}
\caption{ Discontinuity at $\theta=15^{\circ} $}
\label{21a}
\end{subfigure}
\hfill
\begin{subfigure}[b]{0.31\textwidth}
\centering
\includegraphics[width=\textwidth]{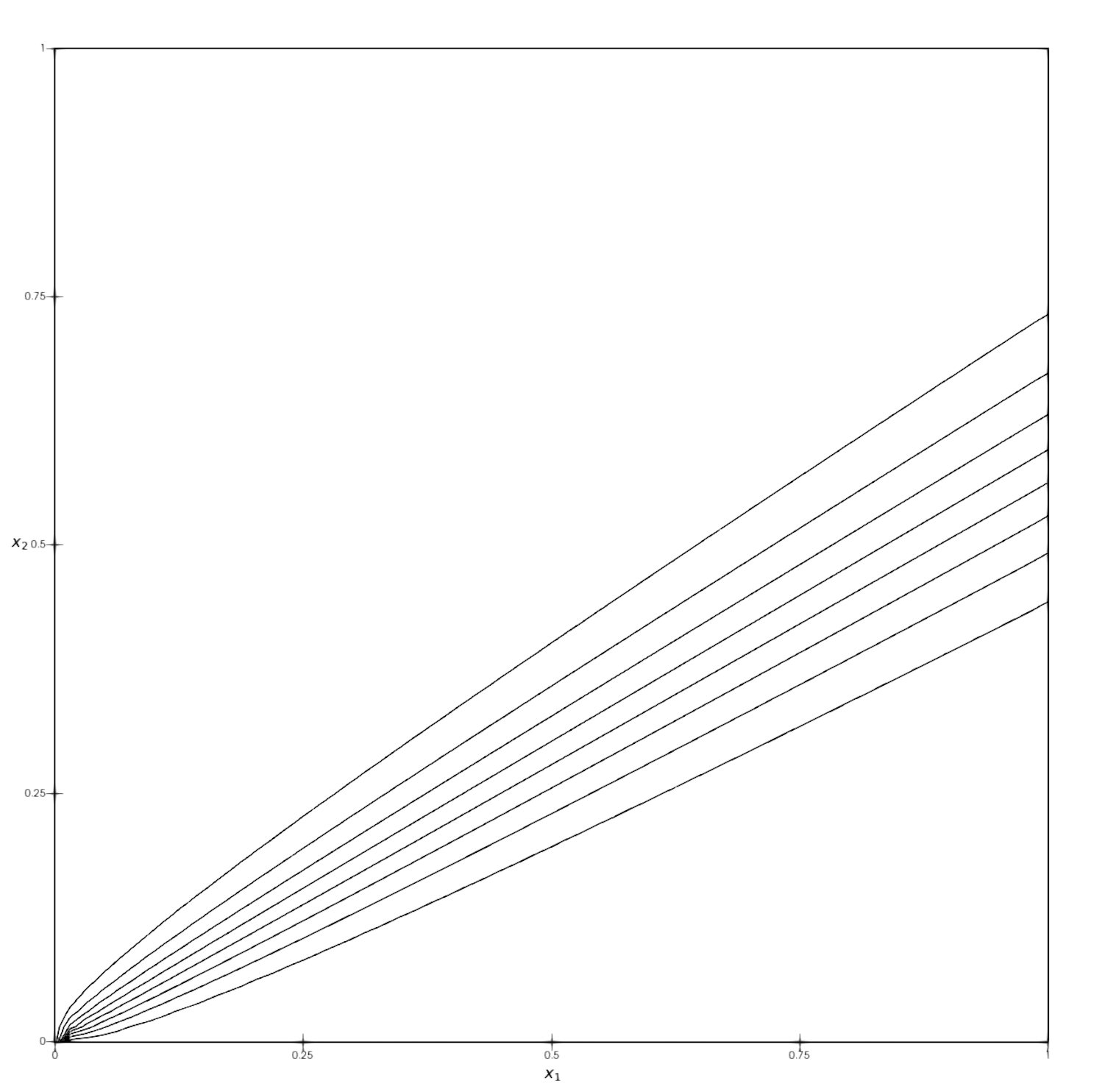}
\caption{ Discontinuity at $\theta=30^{\circ} $}
\label{21b}
\end{subfigure}
\hfill 
\begin{subfigure}[b]{0.31\textwidth}
\centering
\includegraphics[width=\textwidth]{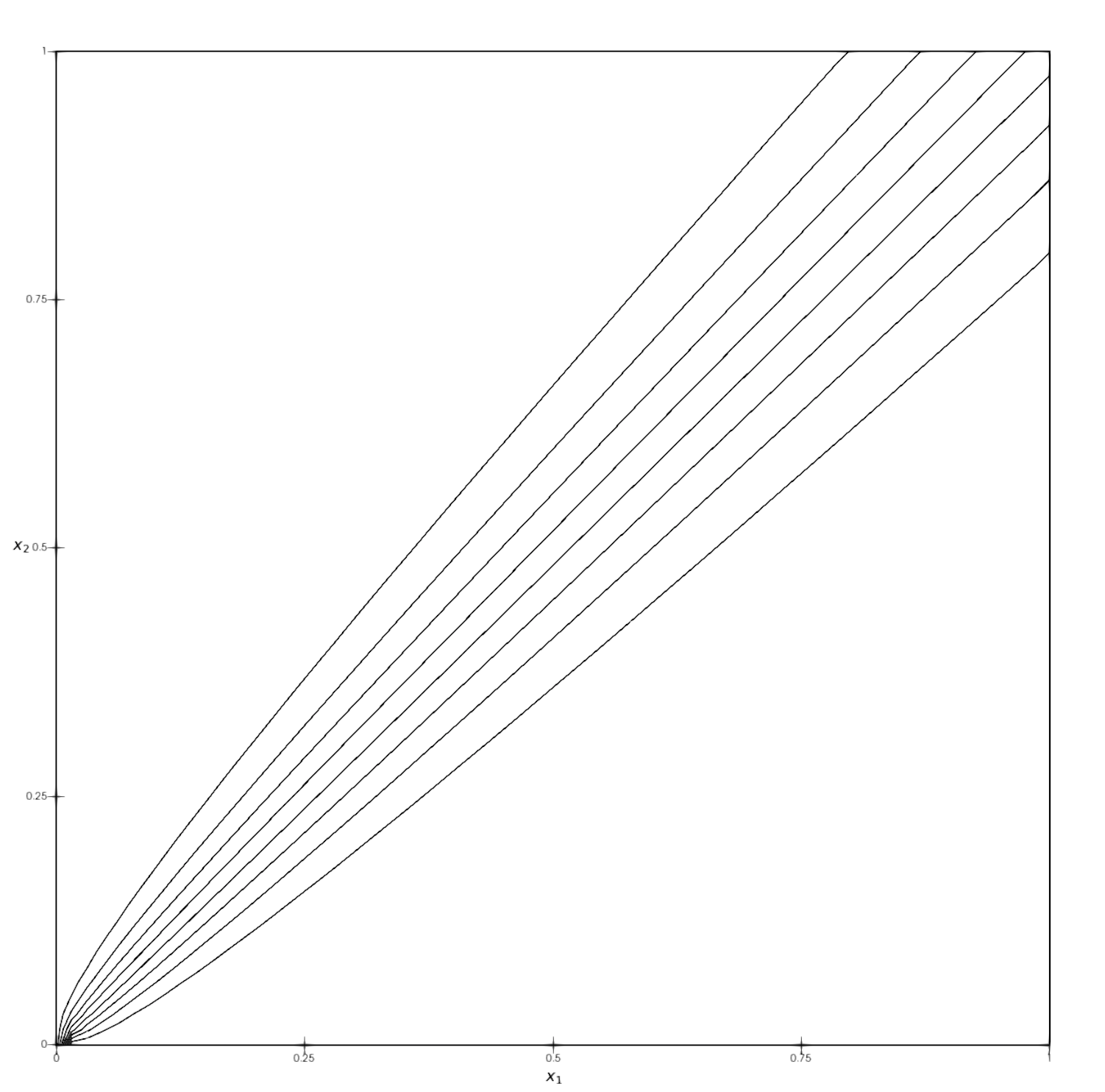}
\caption{ Discontinuity at $\theta=45^{\circ} $}
\label{21c}
\end{subfigure}
\vfill
\begin{subfigure}[b]{0.31\textwidth}
\centering
\includegraphics[width=\textwidth]{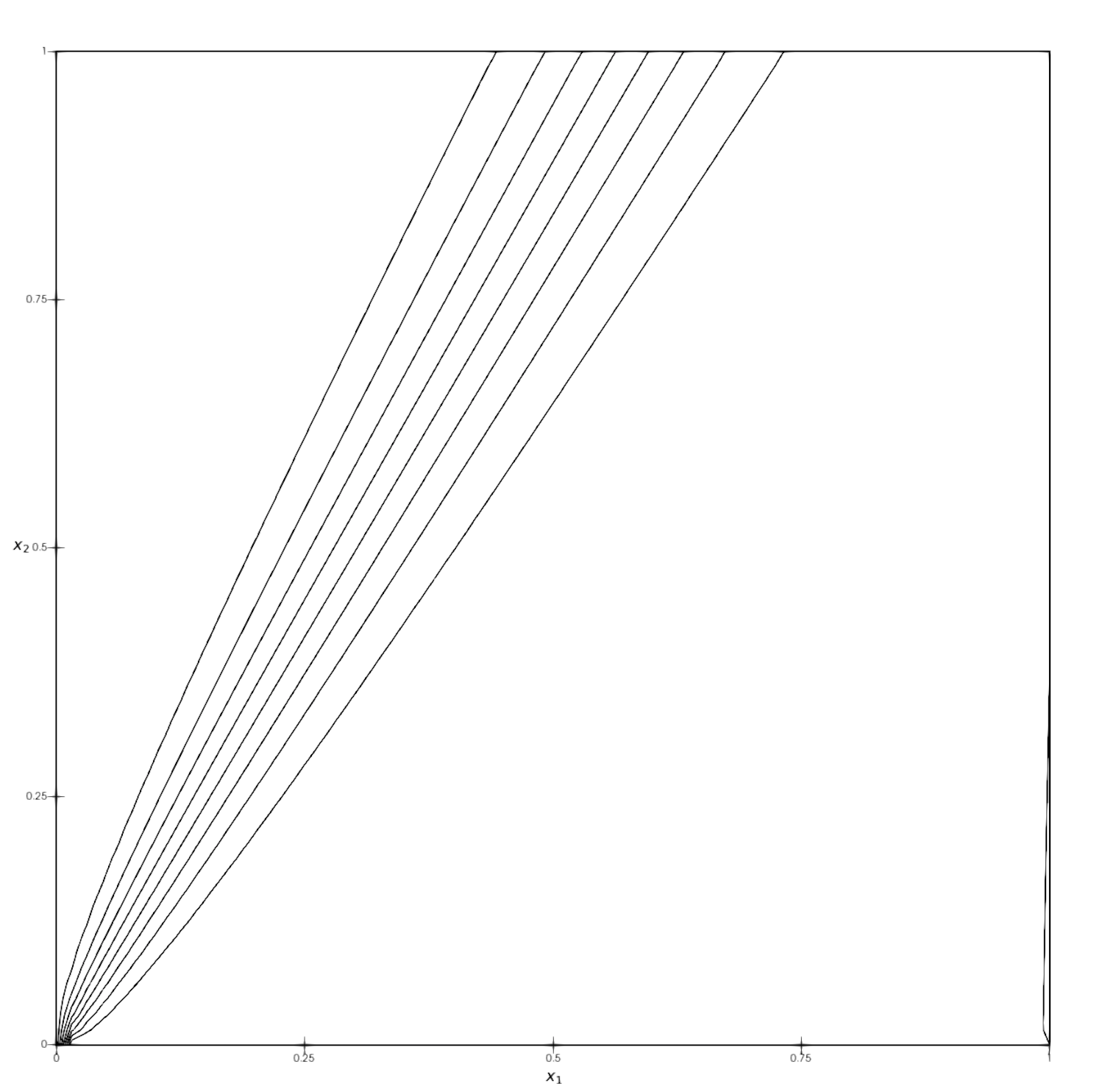}
\caption{ Discontinuity at $\theta=60^{\circ} $}
\label{21d}
\end{subfigure}
\hfill
\begin{subfigure}[b]{0.31\textwidth}
\centering
\includegraphics[width=\textwidth]{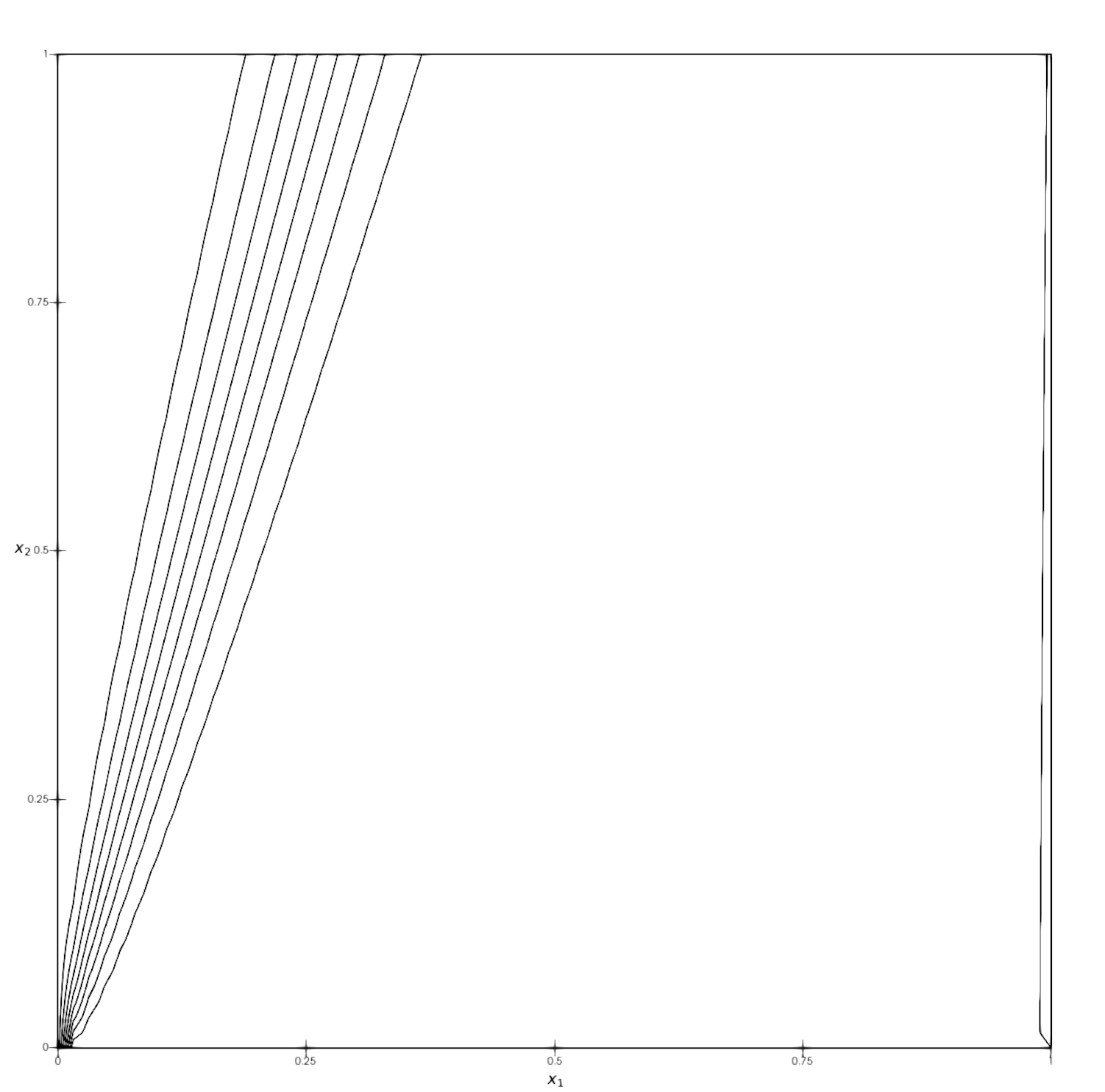}
\caption{ Discontinuity at $\theta=75^{\circ} $}
\label{21e}
\end{subfigure}
\caption{ Discontinuities at various angles: 2D Linear convection equation}
\label{21}
\end{figure} \\
\emph{Semi-circular discontinuities:}
This test case is also from Spekreijse \cite{Spe1987}. Domain is $[-1,1]\times[0,1]$ and $a=x_2, \ b=x_1$. Boundary conditions of the steady-state problem are $u(x_1,0)=0 \ \text{for} \ x_1<-0.65$, $u(x_1,0)=1 \  \text{for} \ -0.35 \geq x_1\geq-0.65$, $u(x_1,0)=0 \ \text{for} \ 0\geq x_1>-0.35$ and  $u(\pm1,x_2)=0 \  \text{for} \ 0<x_2<1$. Exact solution of the steady-state problem is, $u(x_1,x_2)=1 \ \mbox{for} \ 0.35\leq\sqrt{x_1^2+x_2^2}\leq0.65$ and $u(x_1,x_2)=0 \ \mbox{otherwise}$. Numerical solution computed using the formulated LB scheme by splitting the domain into $65 \times 33$ evenly spaced lattice points is shown in \cref{22}. \begin{figure}
\centering
\begin{minipage}{.5\textwidth}
\centering 
\includegraphics[width=0.9\textwidth]{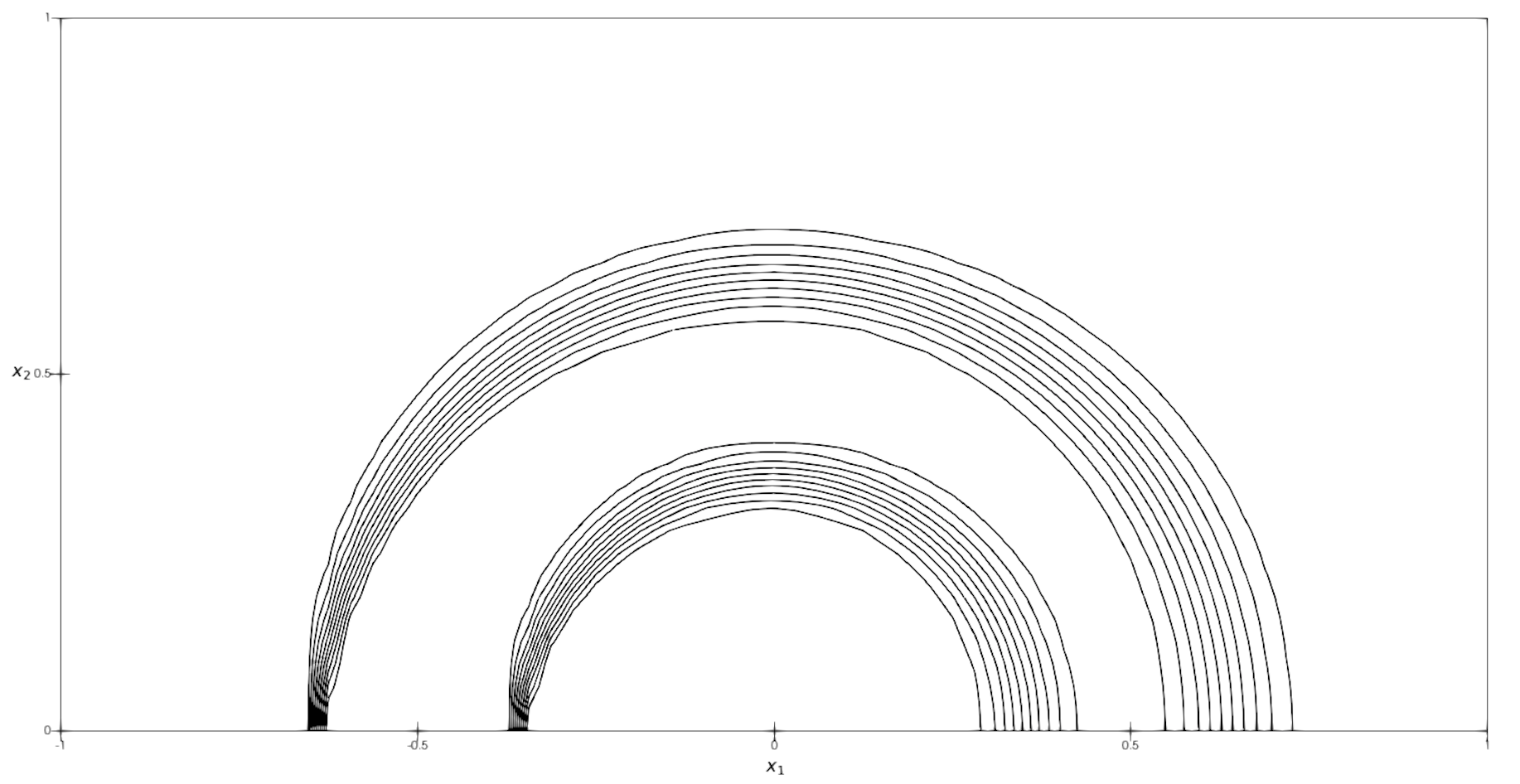}
\captionof{figure}{Semi-circular discontinuities}
\label{22}
\end{minipage}%
\begin{minipage}{.5\textwidth}
\centering
\includegraphics[width=0.9\textwidth]{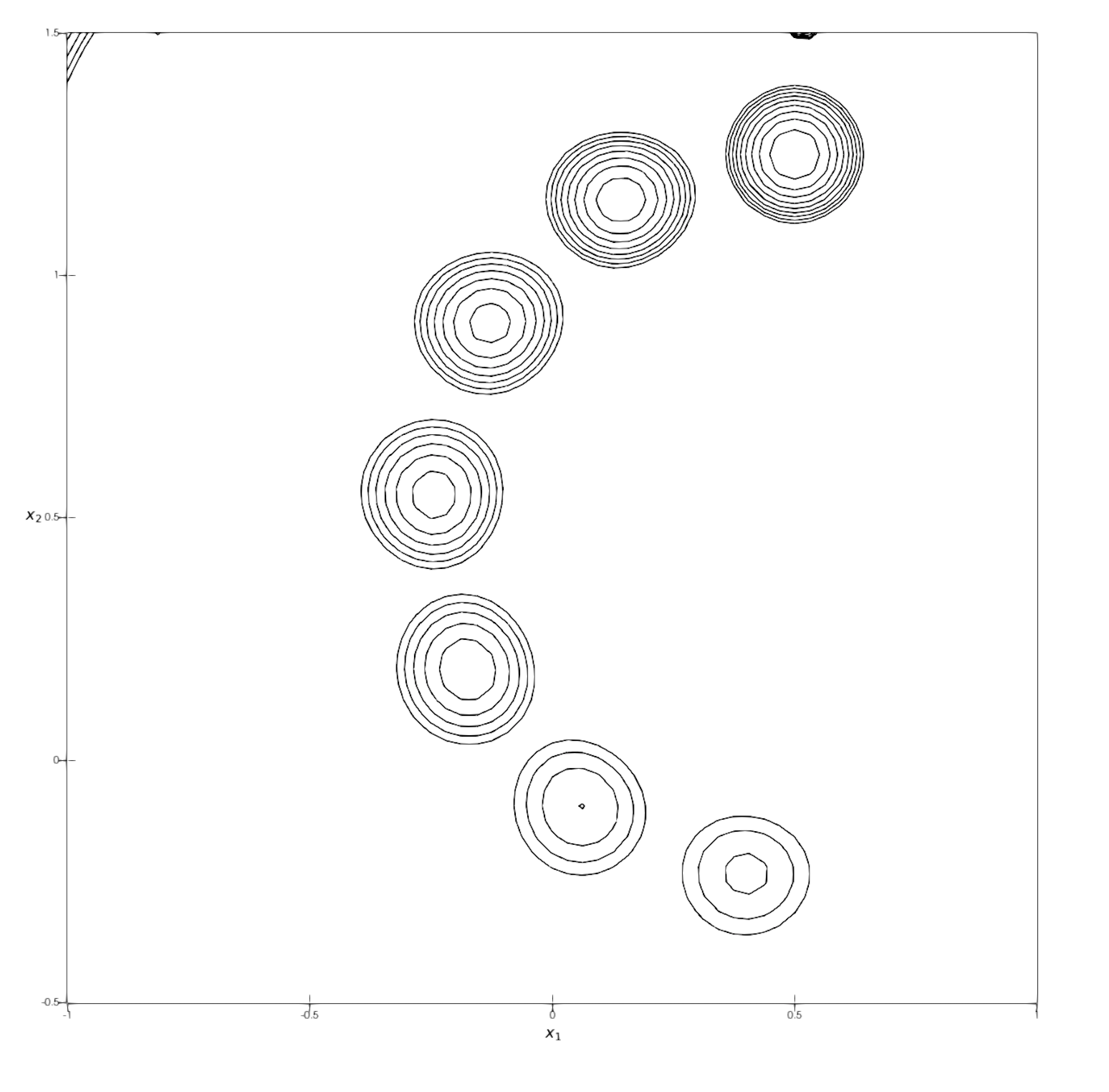}
\captionof{figure}{Solid body rotation}
\label{23}
\end{minipage} 
\end{figure}\\
\emph{Solid body rotation:}
This test case is from LeVeque \cite{L1996}. Here, $a=-(x_2-\frac{1}{2}), \ b=(x_1-\frac{1}{2})$. The initial condition is $u(x_1,x_2,0)=\frac{1}{4} \left( 1+cos(\pi r(x_1,x_2)) \right)$ with $r(x_1,x_2)=min \left( \sqrt{(x_1-x_0)^2+(x_2-y_0)^2},r_0\right)/r_0$ on the domain $[-1,1]\times[-0.5,1.5]$, where $x_0=0.5, \ y_0=1.25 \ \text{and} \ r_0=0.2$. Numerical solutions computed using $65 \times 65$ evenly spaced lattice points on the domain, for $T=0, \ 0.5, \ 1, \ 1.5, \ 2, \ 2.5 \ \text{and} \ 3$ plotted in \cref{23} depict the counter-clockwise rotation of solid body. For all times, $10$ contour levels are sampled out between $u=0.1 \ \text{and} \ u=0.5$. As amplitude reduces over time due to diffusion, the number of contour levels reduce.  

\subsubsection{One dimensional inviscid Burgers' equation}
One dimensional inviscid Burgers' equation, $\partial_t u + \partial_x g(u) = 0 \ \mbox{with} \ g(u)=\frac{1}{2}u^2$, is solved using the formulated LB scheme for various boundary and initial conditions. \\
\emph{Initial sine wave profile:}
This test case is from Ben-Artzi \& Falcovitz \cite{BF2003}. Position variable $x \in [0,1]$ and initial condition is $u(x,0)=sin(2\pi x)$. Exact solution is found using the method of characteristics. For numerical solution using upwind LB scheme, the domain [0,1] is split up into 81 evenly spaced lattice points and the LB algorithm is followed.  \Cref{1} shows the sequential behaviour of $u$ over time $t$. The initial sine wave profile at $t=0$ is shown in \cref{1a}. A smooth compressed profile formed at $t=\frac{0.8}{2\pi}$ is shown in \cref{1b}. Over time, the profile compresses and forms a discontinuity at $t=\frac{1}{2\pi}$ as shown in \cref{1c}. The jump discontinuity at $x=0.5$ gradually increases and becomes a jump between $u=1 \ \mbox{and} \ u=-1$ at $t=0.25$ as shown in \cref{1d}. At this time, the propagation speed of shock vanishes and standing shock is formed. At later times, the jump discontinuity decreases and becomes a jump between $u<1 \ \mbox{and} \ u>-1$ as shown in \cref{1e,1f}. 
\begin{figure}
\centering
\begin{subfigure}[b]{0.47\textwidth}
\centering
\includegraphics[width=\textwidth]{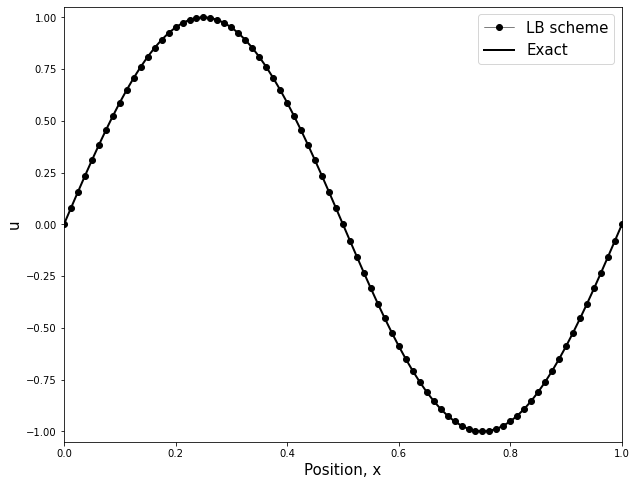}
\caption{ Sine wave profile at $t=0$}
\label{1a}
\end{subfigure}
\hfill
\begin{subfigure}[b]{0.47\textwidth}
\centering
\includegraphics[width=\textwidth]{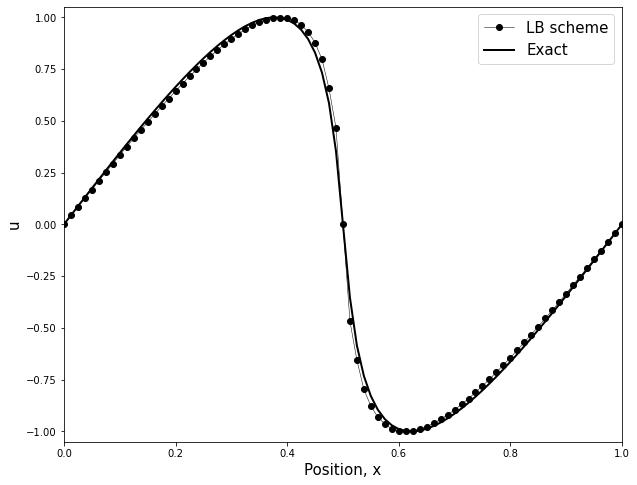}
\caption{ Smooth profile at $t=\frac{0.8}{2\pi}$}
\label{1b}
\end{subfigure}
\vfill 
\begin{subfigure}[b]{0.47\textwidth}
\centering
\includegraphics[width=\textwidth]{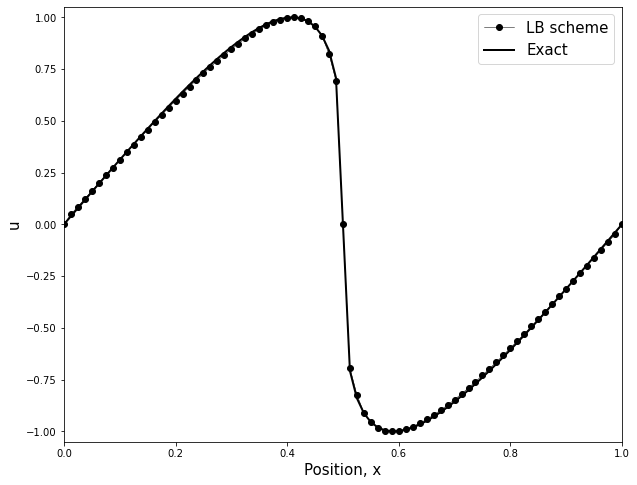}
\caption{ Discontinuous profile at $t=\frac{1}{2\pi}$}
\label{1c}
\end{subfigure}
\hfill
\begin{subfigure}[b]{0.47\textwidth}
\centering
\includegraphics[width=\textwidth]{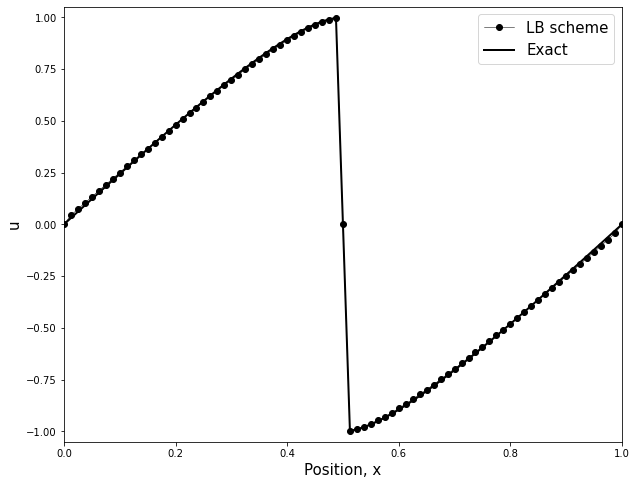}
\caption{ Discontinuous profile at $t=0.25$}
\label{1d}
\end{subfigure}
\vfill
\begin{subfigure}[b]{0.47\textwidth}
\centering
\includegraphics[width=\textwidth]{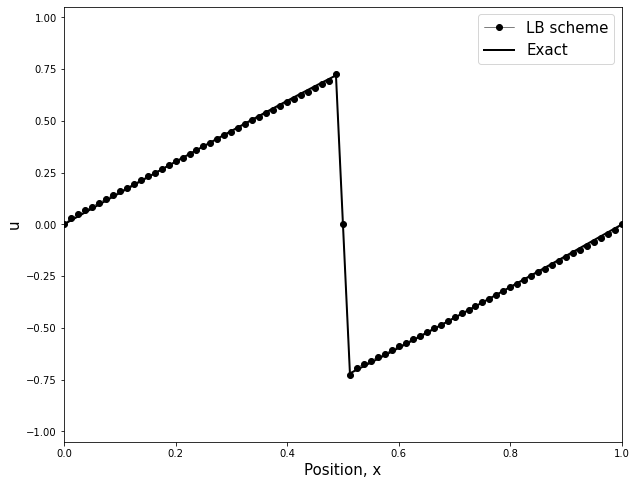}
\caption{ Standing shock at $t=0.5$}
\label{1e}
\end{subfigure}
\hfill  
\begin{subfigure}[b]{0.47\textwidth}
\centering
\includegraphics[width=\textwidth]{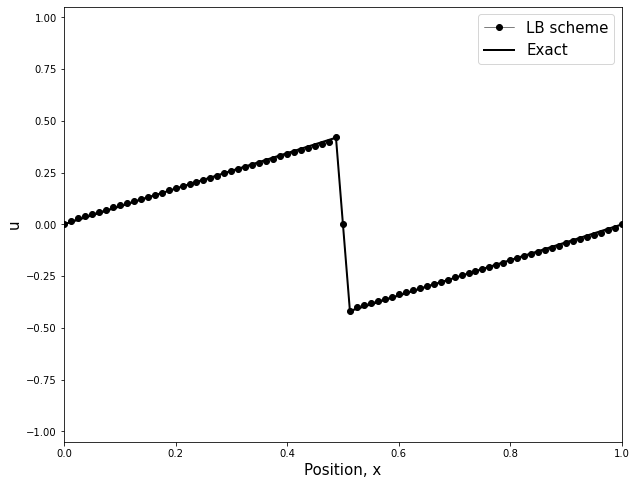}
\caption{ Standing shock at $t=1$}
\label{1f}
\end{subfigure}
\caption{ Initial Sine wave profile for inviscid Burgers' equation}
\label{1}
\end{figure}\\
\emph{Initial square wave profile without a sonic point:}
This test case is from Laney \cite{Lan1998}. Domain of the problem is $[-1,1]$ and initial condition is the square wave, $u(x,0)= 1 \ \text{for} \  |x| \leq \frac{1}{3}$ and $u(x,0)= 0 \  \text{for} \  \frac{1}{3}<|x|\leq1$. Exact solution is found using the method of characteristics. For numerical solution, the domain is split up into 41 evenly spaced lattice points. The jump from $u=0 \ \mbox{to} \ u=1$ at $x=-\frac{1}{3}$ becomes an expansion fan while the jump from $u=1 \ \mbox{to} \ u=0$ at $x=\frac{1}{3}$ becomes a shock. The solution at $t=0.6$ is shown in \cref{2}. 
\begin{figure}
\centering
\begin{minipage}{.45\textwidth}
\centering
\includegraphics[width=0.9\textwidth]{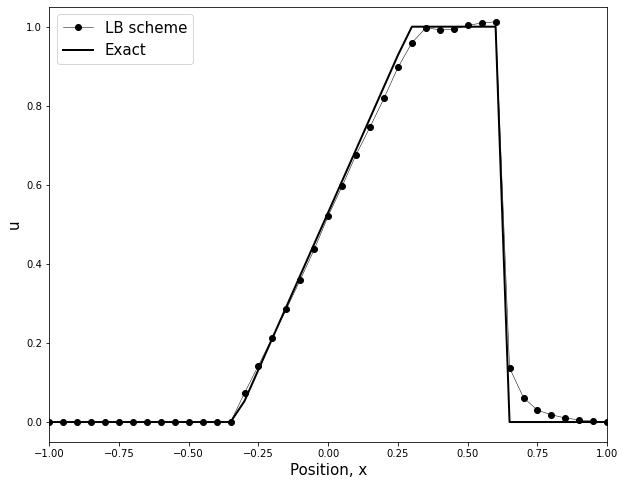}
\captionof{figure}{Initial square wave profile without a sonic point}
\label{2}
\end{minipage}% 
\hfill 
\begin{minipage}{.45\textwidth}
\centering
\includegraphics[width=0.9\textwidth]{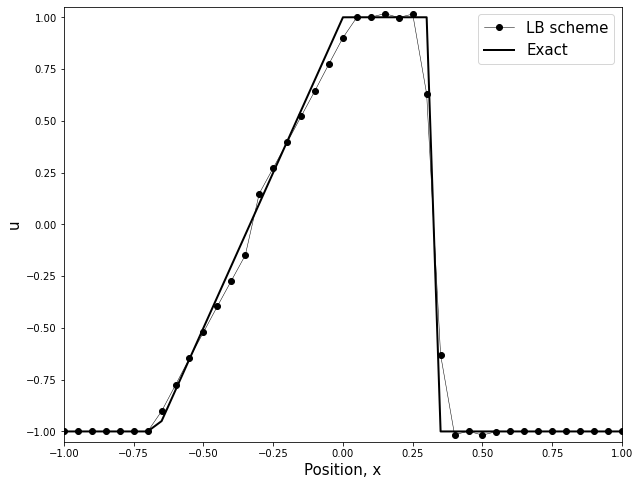}
\captionof{figure}{Initial square wave profile with a sonic point}
\label{3}
\end{minipage}
\end{figure}\\
\emph{Initial square wave profile with a sonic point:}
This test case is also from Laney \cite{Lan1998}. Domain of the problem is $[-1,1]$ and initial condition is the square wave, $u(x,0)=  1 \ \text{for} \  |x| \leq \frac{1}{3}$ and $u(x,0)= -1 \  \text{for} \  \frac{1}{3}<|x|\leq1$. Here again, the exact solution is found using method of characteristics. The domain is split up into 41 evenly spaced lattice points for finding the numerical solution. The jump from $u=-1 \ \mbox{to} \ u=1$ at $x=-\frac{1}{3}$ becomes an expansion fan with a sonic point while the jump from $u=1 \ \mbox{to} \ u=-1$ at $x=\frac{1}{3}$ becomes a stationary shock. Expansion fan and shock wave symmetrically span the sonic point $u=0$. The solution at $t=0.3$ is shown in \cref{3}. 

\subsubsection{Two dimensional non-linear scalar conservation law}
The two dimensional non-linear scalar conservation law is $\partial_t u+ \partial_{x_1} g_1(u) + \partial_{x_2} g_2(u) = 0$ with $g_1(u)$ and $g_2(u)$ being non-linear functions of $u$. This equation is solved for various boundary and initial conditions using the upwind LB scheme.
\begin{figure}[h!] 
\centering
\begin{minipage}{.4\textwidth}
\centering
\includegraphics[width=0.9\textwidth]{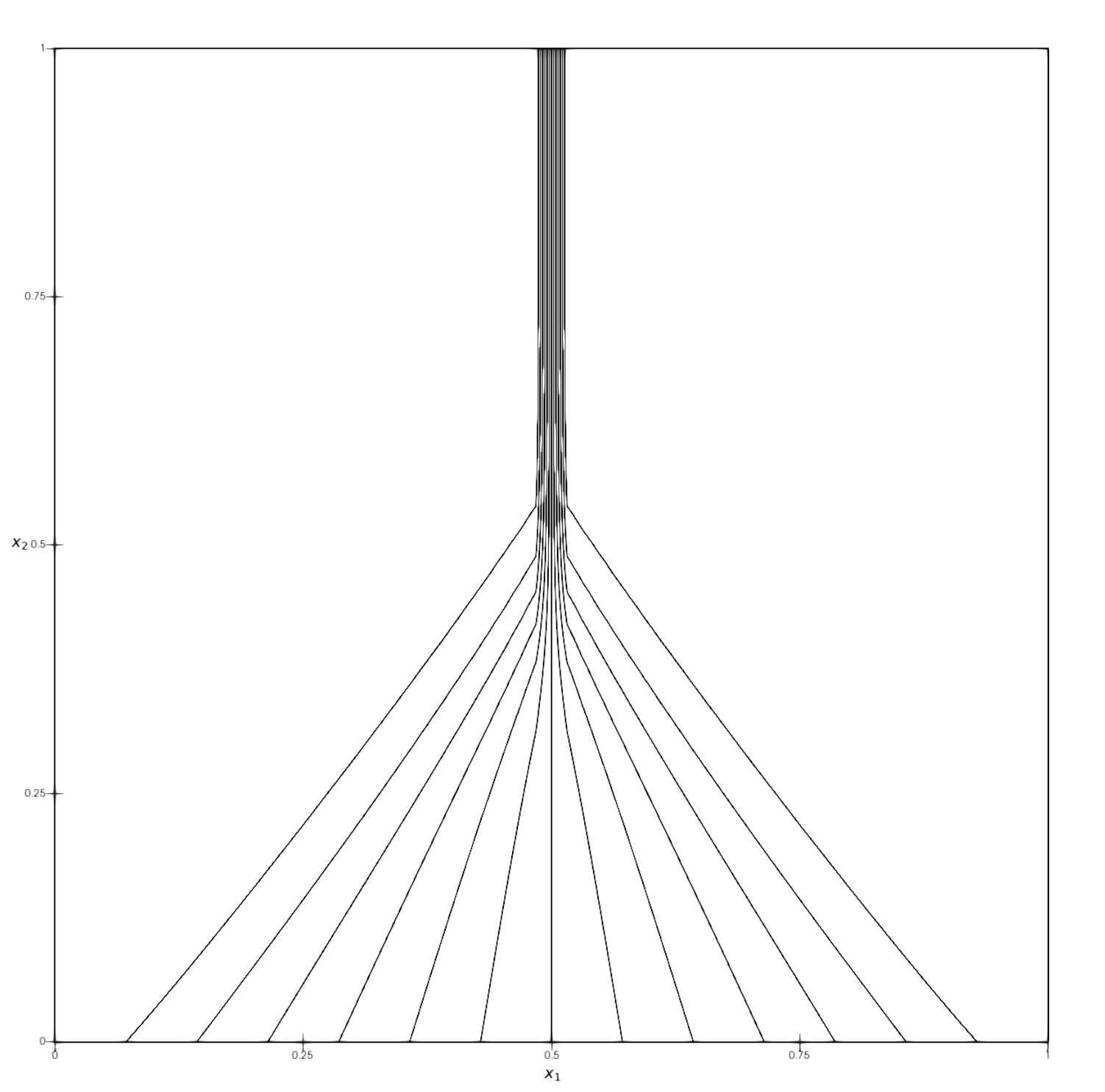}
\captionof{figure}{Normal shock}
\label{25}
\end{minipage}%
\begin{minipage}{.4\textwidth}
\centering
\includegraphics[width=0.9\textwidth]{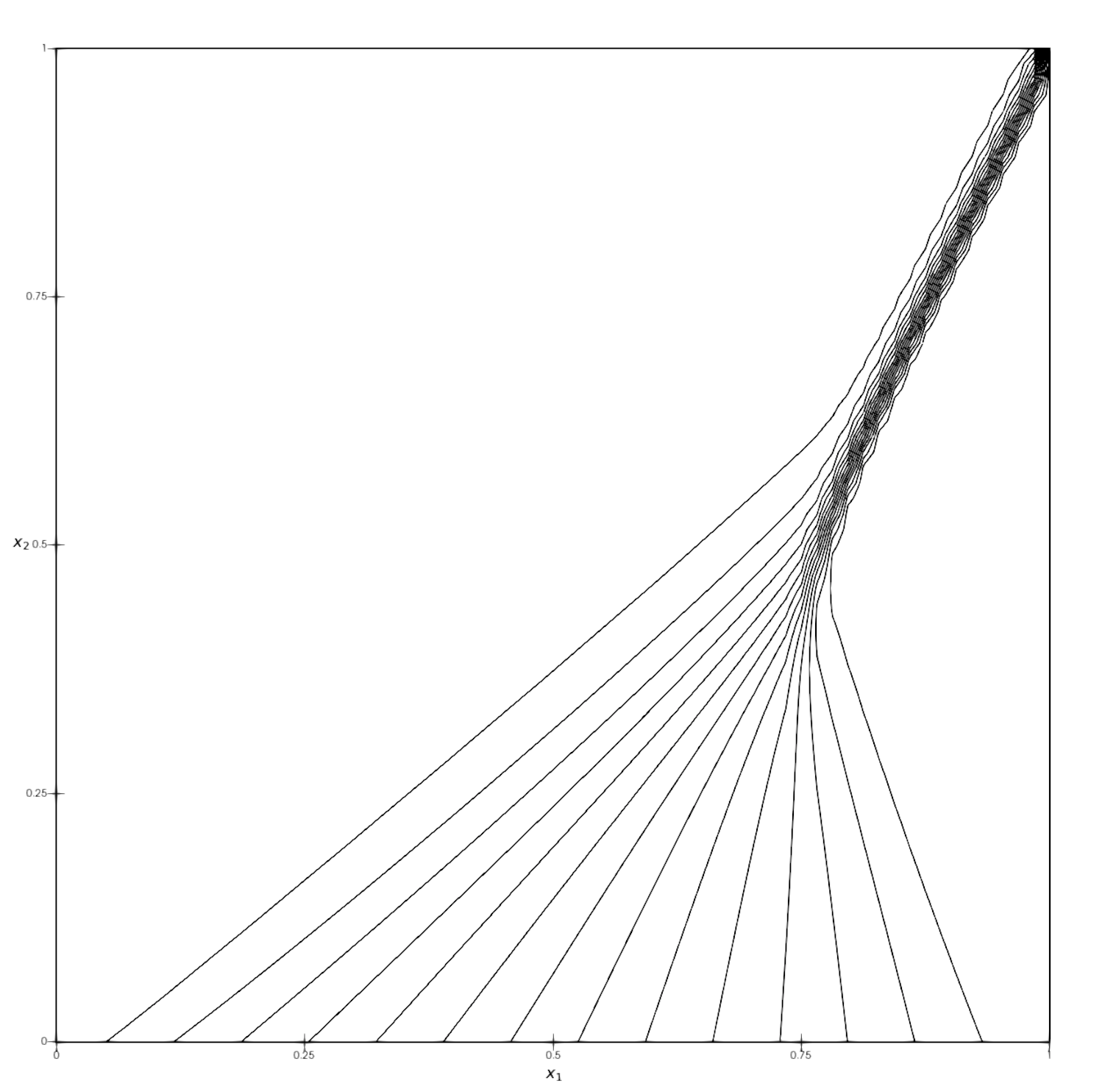}
\captionof{figure}{Oblique shock}
\label{26}
\end{minipage}
\end{figure}\\
\emph{Normal shock:}
This test case is from Spekreijse \cite{Spe1987}. Domain of the problem is $[0,1]\times[0,1]$ and the flux functions are $g_1(u)=\frac{1}{2}u^2$ and $g_2(u)=u$, with nonlinearity in $g_1(u)$. Boundary conditions for the steady-state problem are $u(0,x_2)=1 \ \text{for} \  0<x_2<1$, $u(1,x_2)=-1 \ \mbox{for} \  0<x_2<1$ and $u(x_1,0)=1-2x_1 \ \mbox{for} \ 0<x_1<1$. Numerical solution computed by splitting the domain into $65 \times 65$ evenly spaced lattice points is shown in \cref{25}. \\ 
\emph{Oblique shock:} 
This test case is also from Spekreijse \cite{Spe1987}. Domain of the problem is $[0,1]\times[0,1]$ and the flux functions are $g_1(u)=\frac{1}{2}u^2$ (nonlinear) and $g_2(u)=u$. Boundary conditions for the steady-state problem are $u(0,x_2)=1.5   \ \mbox{for} \   0<x_2<1$,  $u(1,x_2)=-0.5  \ \mbox{for} \  0<x_2<1$ and $u(x_1,0)=1.5-2x_1 \ \mbox{for} \  0<x_1<1$. Numerical solution computed by splitting the domain into $65 \times 65$ evenly spaced lattice points is shown in \cref{26}. 

\section{Extension of upwind LB scheme to conservation laws with source terms}
Development of numerical methods for conservation laws with source terms are nontrivial, especially when the source terms are stiff. In this section, the formulated upwind LB scheme will be extended to hyperbolic conservation laws with source terms. Consider the conservation law,
\begin{equation}
\label{Scalar cons. law source}
 \partial_t u + \sum_{d=1}^D \partial_{x_d} g_d (u) = s(u) \ \text{with IC} \ u(\mathbf{x},0)=u_0(\mathbf{x})
\end{equation}
Here $s(u)$ is the source term. The definitions of $u$ and $g_d (u)$ follow from the previous section. The problem in \cref{Scalar cons. law source} is solved approximately using the LBE,
\begin{multline}
\label{LBE source}
f_n(\mathbf{x}+\mathbf{v_n}\Delta t, t+\Delta t)=(1-\omega)f_n(\mathbf{x}, t) + \omega f_n^{eq} (u(\mathbf{x}, t)) + \\ \frac{1}{2}\biggl( r_n(u(\mathbf{x}+\mathbf{v_n}\Delta t, t+\Delta t)) + r_n(u(\mathbf{x}, t)) \biggr) \ \text{for} \ n \in \{1,2,..,N \}
\end{multline}
Here $r_n:\mathcal{U} \rightarrow \mathbb{R}$ models source term at the mesoscopic level. The definitions of $f_n, \ f_n^{eq} \ \text{and} \ \omega$ follow from the previous section. The LBE in \cref{LBE source} can be split up into,
\begin{align}
\mbox{Collision:} \quad \mathcal{F}_n^{*}=(1-\omega)f_n(\mathbf{x}, t) + \omega f_n^{eq} (u(\mathbf{x}, t)) + \frac{1}{2}r_n(u(\mathbf{x}, t))\label{collision source} \\ \mbox{Streaming:} \quad\quad \quad \quad \quad \quad \quad \quad \quad \quad \quad \mathcal{F}_n(\mathbf{x}+\mathbf{v_n}\Delta t, t+\Delta t)=\mathcal{F}_n^{*} \label{streaming source}
\end{align}
Here $\mathcal{F}_n=f_n-\frac{1}{2}r_n$. Taylor expanding the LBE in \cref{LBE source} and simplifying, we get
\begin{multline}
\label{TE2 source}
\left( \partial_t + \sum_{d=1}^D v_n^{(d)} \partial_{x_d}\right) f_n =  - \frac{1}{\epsilon} \left(f_n-f_n^{eq}\right) \\ + \frac{\Delta t}{2\epsilon} \left( \partial_t + \sum_{d=1}^D v_n^{(d)} \partial_{x_d}\right) (f_n-f_n^{eq}) + \frac{1}{\Delta t}r_n + O(\Delta t^2)
\end{multline}
Using multiple scale expansion of $r_n$ as $r_n = \xi r_n^{(1)}+\xi^2 r_n^{(2)}+...$ along with the multiple scale expansions of $f_n$ and derivatives of $f_n$ in \cref{TE2 source} and doing Chapman-Enskog analysis, the following modified PDE is obtained. 
\begin{multline}
\label{mPDE2 source}
\partial_t u + \sum_{d=1}^D \partial_{x_d} g_d(u) = \sum_{n=1}^N \frac{1}{\Delta t}r_n + \\  \Delta t \left( \frac{1}{\omega}-\frac{1}{2}\right) \sum_{d=1}^D \partial_{x_d} \left(  \sum_{i=1}^D \partial_{x_i} \left( \sum_{n=1}^N v_n^{(d)} v_n^{(i)} f_n^{eq} \right) - \partial_u g_d \left( \sum_{i=1}^D \partial_u g_i \partial_{x_i} u \right)\right) \\
+ \Delta t \left( \frac{1}{\omega}-\frac{1}{2}\right)  \sum_{d=1}^D  \partial_{x_d} \left( \partial_u g_d \sum_{n=1}^N \frac{1}{\Delta t} r_n^{(1)} - \sum_{n=1}^N \frac{1}{\Delta t} v_n^{(d)} r_n^{(1)}\right) + O(\Delta t^2)
\end{multline}
$r_n$ must be chosen such that the first term in RHS of \cref{mPDE2 source} is $s(u)$. The middle and last terms in RHS of \cref{mPDE2 source} represent numerical diffusion due to convection terms ($g_d(u)$) and spurious numerical convection due to source term ($s(u)$) respectively. Numerical diffusion due to convection terms dictates the choice of $\lambda$ and $\omega$ as discussed in the previous section. The spurious numerical convection due to source term can be nullified if $r_n$ is chosen such that $\partial_u g_d \sum_{n=1}^N \frac{1}{\Delta t} r_n^{(1)} = \sum_{n=1}^N \frac{1}{\Delta t} v_n^{(d)} r_n^{(1)}, \ \forall d$. For this, we require $r_n$ to be, 
 \begin{align}
r_n(u)&= \Delta t \frac{s(u) \partial_u g_n(u)^+}{\lambda_d} \ ,\quad\text{if} \ n \in \{1,2,...,D \}  \label{source distribution functions_a} \\ r_{D+1}(u)&=\Delta t s(u) \left( 1-\left(\sum_{d=1}^D \frac{\partial_u g_d(u)^+ + \partial_u g_d(u)^-}{\lambda_d}\right)\right)  \label{source distribution functions_b} \\ r_n(u)&=\Delta t \frac{s(u) \partial_u g_{n-(D+1)}(u)^-}{\lambda_{n-(D+1)}} \ ,\quad\text{if} \ n \in \{D+2,...,2D+1 \}  \label{source distribution functions_c}
  \end{align}
  These satisfy the moments $\sum_{n=1}^N \frac{1}{\Delta t}r_n = s(u)$ and $\partial_u g_d \sum_{n=1}^N \frac{1}{\Delta t} r_n =$ $\sum_{n=1}^N \frac{1}{\Delta t} v_n^{(d)} r_n$ $\forall d$, and hence the spurious numerical convection due to source term has been removed. \Cref{alg:LB source} can be followed to solve \cref{LBE source}.
\begin{algorithm}[h!] 
\caption{LB algorithm for conservation laws with source terms}
\label{alg:LB source}
\begin{algorithmic}[1]
\STATE{Evaluate $f_n^{eq}(\mathbf{x},0)$ and $r_n(\mathbf{x},0)$ from $u_0(\mathbf{x})$ using \cref{Flux decomposed equilibrium distribution functions_a,Flux decomposed equilibrium distribution functions_b,Flux decomposed equilibrium distribution functions_c} and \cref{source distribution functions_a,source distribution functions_b,source distribution functions_c} respectively, for $n \in \{1,2,..,N\}$ and $\forall \mathbf{x}$ in the lattice.}
\STATE{Initialise $f_n(\mathbf{x},0)=f_n^{eq}(\mathbf{x},0)$ for $n \in \{1,2,..,N\}$ and $\forall \mathbf{x}$ in the lattice, and take $t=0$.}
\WHILE{$T-t>10^{-8}$}
\STATE{Carry out Collision step using \cref{collision source} for an appropriate choice of $\omega$.}
\STATE{Carry out Streaming step using \cref{streaming source} for appropriate choice of discrete velocities $\mathbf{v_n}$ such that \cref{Discrete velocities_a,Discrete velocities_b,Discrete velocities_c} hold. For fixed $\lambda_d=\lambda>0$ where $d \in \{ 1,2,..,D\}$ and uniform lattice spacing, $\Delta t = \frac{\Delta x_d}{\lambda_d}=\frac{\Delta x}{\lambda}$ is determined uniquely for any $d$. }
\STATE{Evaluate $\sum_{n=1}^N \mathcal{F}_n$ at all interior points and use Newton's root finding method to find $u$ from $\sum_{n=1}^N\mathcal{F}_n=u-\frac{\Delta t}{2}s(u)$.}
\STATE{Evaluate $f_n^{eq}(u), \ g_d(u), \ s(u) \ \text{and} \ r_n(u)$ $\forall n \in \{ 1,2,..,N\}$ at all interior points.}\STATE{Evaluate $f_n$ using $f_n=\mathcal{F}_n-\frac{1}{2}r_n \ \forall n \in \{ 1,2,..,N\}$ at all interior points.}
\STATE{Use appropriate boundary conditions for $f_n$.}
\STATE{Update $t=t+\Delta t$}
\ENDWHILE
\RETURN $u$
\end{algorithmic}
\end{algorithm} 

\subsection{Numerical results and discussions}
Numerical results for some standard test problems with source terms are discussed in this section. 
\subsubsection{LeVeque and Yee's problem}
This is the test case used by LeVeque and Yee \cite{LY1990} to understand the cause for incorrectness in speeds of discontinuities. Governing equation is $\partial_t u +\partial_x u = -\mu u (u-1) (u-\frac{1}{2})$ in the domain $0\leq x \leq 1$. Initial conditions are $u(x,0)=1 \ \text{for} \ x\leq0.3$ and $u(x,0)=0 \ \text{for} \ x>0.3$. To obtain the numerical solution, domain is split up into 51 evenly spaced points. A comparison of numerical solutions reproduced from LeVeque and Yee \cite{LY1990} and numerical solutions from LB scheme is shown in \cref{S} at $T=0.3$ for different values of $\mu$. It is evident that the formulated upwind LB scheme is devoid of the effects of spurious numerical convection due to source terms, thereby capturing the discontinuities at correct locations, despite high stiffness in source terms. 
\begin{figure}[h!]
\centering
\begin{subfigure}[b]{0.4\textwidth}
\centering
\includegraphics[width=\textwidth]{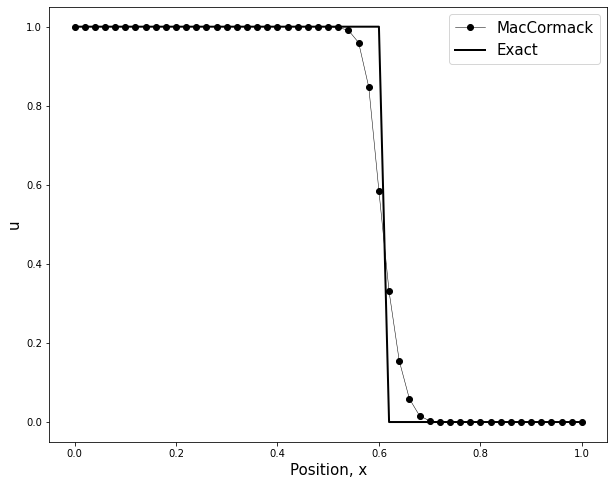}
\caption{ $\mu = 1$}
\end{subfigure}
\hfill
\begin{subfigure}[b]{0.4\textwidth}
\centering
\includegraphics[width=\textwidth]{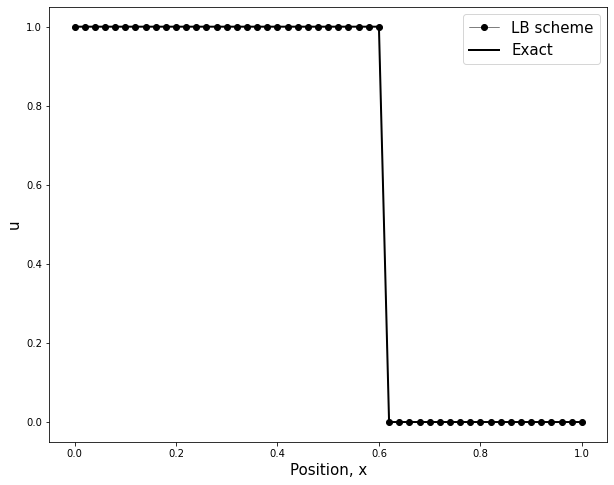}
\caption{ $\mu = 1$}
\end{subfigure}
\vfill 
\begin{subfigure}[b]{0.4\textwidth}
\centering
\includegraphics[width=\textwidth]{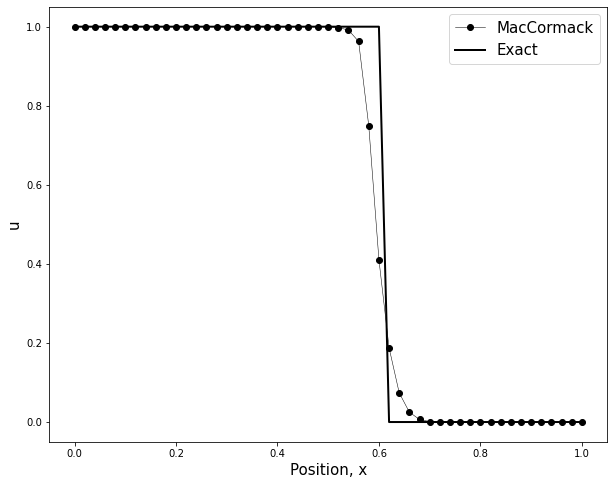}
\caption{ $\mu=10$}
\end{subfigure}
\hfill
\begin{subfigure}[b]{0.4\textwidth}
\centering
\includegraphics[width=\textwidth]{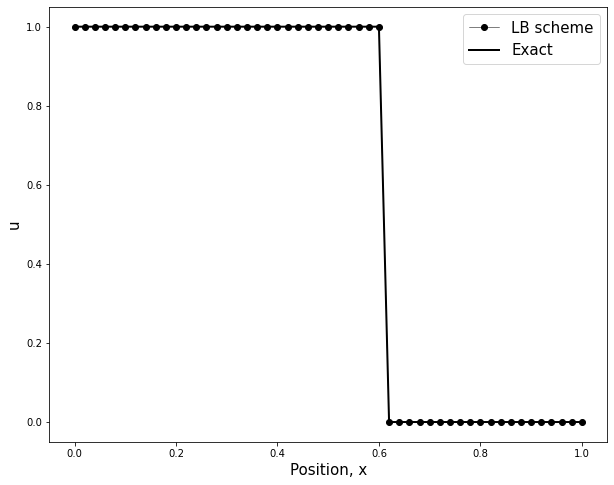}
\caption{ $\mu=10$}
\end{subfigure}
\vfill
\begin{subfigure}[b]{0.4\textwidth}
\centering
\includegraphics[width=\textwidth]{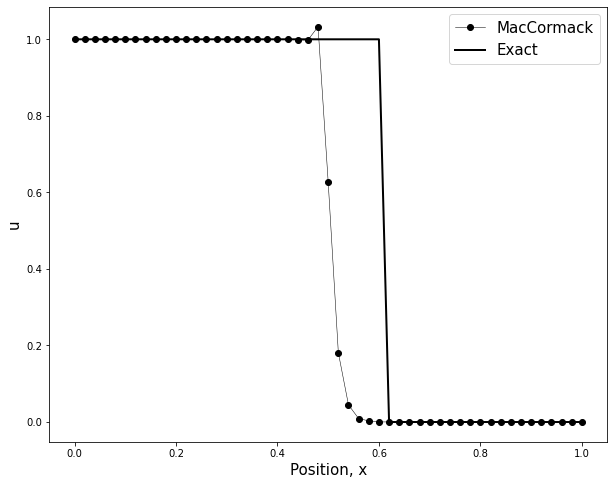}
\caption{ $\mu=100$}
\end{subfigure}
\hfill
\begin{subfigure}[b]{0.4\textwidth}
\centering
\includegraphics[width=\textwidth]{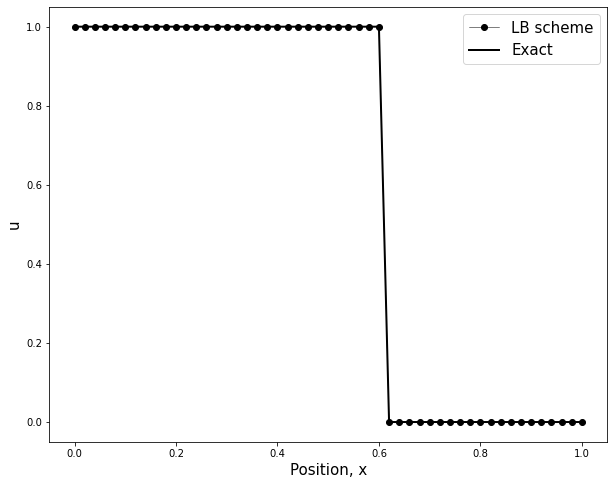}
\caption{ $\mu=100$}
\end{subfigure}
\vfill
\begin{subfigure}[b]{0.4\textwidth}
\centering
\includegraphics[width=\textwidth]{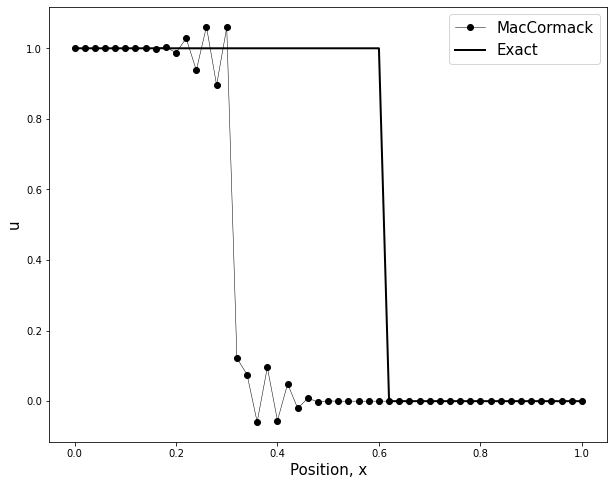}
\caption{ $\mu=1000$}
\end{subfigure}
\hfill
\begin{subfigure}[b]{0.4\textwidth}
\centering
\includegraphics[width=\textwidth]{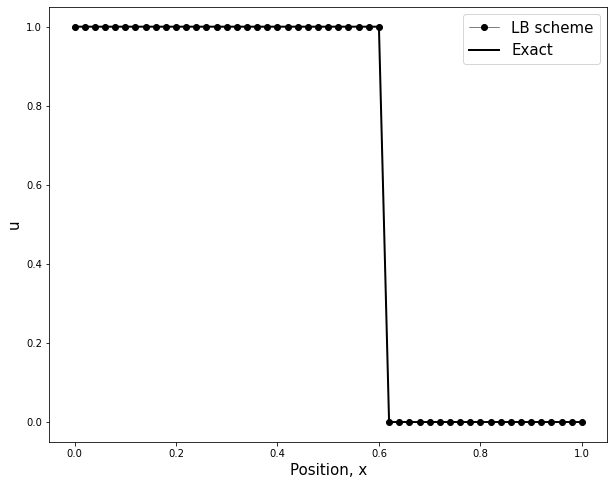}
\caption{ $\mu=1000$}
\end{subfigure}
\caption{ Left: Extended MacCormack's method with limiter based on $u^n$(Reproduced from \cite{LY1990}), Right: Formulated LB scheme for conservation laws with source terms}
\label{S}
\end{figure}

\subsubsection{Embid problem}
This test case is from Embid, Goodman and Majda \cite{EGM1984}. Governing equation is the non-linear scalar conservation law $\partial_t u + u \partial_x u = (6x-3)u$ in the domain $0\leq x\leq1$. Initial conditions are $u(x,0)=1 \ \text{for} \ x\leq0.18$ and $u(x,0)=-0.1 \ \text{for} \ x>0.18$. Boundary conditions are $u(0,t)=1 \ \text{and} \ u(1,t)=-0.1$. For numerical solution of this steady-state problem using LB scheme, the domain $[0,1]$ is split up into 41 evenly spaced lattice points. Numerical solution obtained using LB scheme plotted against the exact solution is shown in \cref{fig:S1}. 
\begin{figure}[h!]
\centering
\includegraphics[scale=0.4]{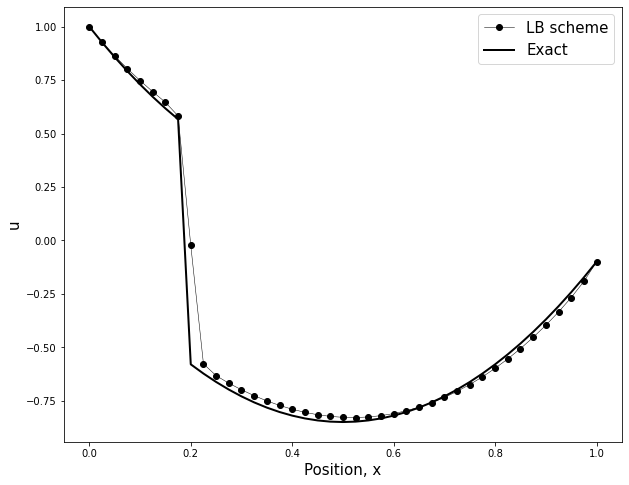}
\caption{Embid problem}
\label{fig:S1}
\end{figure} 

\section{Conclusions}
The LB scheme formulated using flux decomposed equilibrium distribution functions is equivalent to Engquist-Osher scheme upto second order in time. Hence, this LB scheme corresponds to an upwind scheme at the macroscopic level. In the extension of LB scheme to scalar conservation laws with source terms, the functions that model source term at the mesoscopic level are in the same form as flux decomposed equilibrium distribution functions, and they remove the spurious numerical convection occurring due to source terms. This is evident from the comparison of numerical results from Extended MacCormack's method and formulated LB scheme for LeVeque and Yee \cite{LY1990}'s test problem, where despite the stiff source terms, the discontinuities are captured at correct locations.

\bibliographystyle{siam}
\bibliography{references}  

\begin{thebibliography}{10}

\bibitem{ADN2000}
{\sc D.~Aregba-Driollet and R.~Natalini}, {\em Discrete kinetic schemes for
  multidimensional systems of conservation laws}, SIAM J. Numer. Anal., 37
  (2000), pp.~1973--2004.

\bibitem{BF2003}
{\sc M.~Ben-Artzi and J.~Falcovitz}, {\em Generalized riemann problems in
  computational fluid dynamics}.
\newblock Cambridge University Press, 2003.

\bibitem{CS2017}
{\sc T.~Chen and C.-W. Shu}, {\em Entropy stable high order discontinuous
  galerkin methods with suitable quadrature rules for hyperbolic conservation
  laws}, J. Comput. Phys., 345 (2017), pp.~427--461.

\bibitem{EO}
{\sc B.~Engquist and S.~Osher}, {\em Stable and entropy satisfying
  approximations for transonic flow calculations}, Math. Comp., 34 (1980),
  pp.~45--75.

\bibitem{Lan1998}
{\sc C.~B. Laney}, {\em Computational gasdynamics}.
\newblock Cambridge University Press, 1998.

\bibitem{L1996}
{\sc R.~J. LeVeque}, {\em High-resolution conservative algorithms for advection
  in incompressible flow}, SIAM J. Numer. Anal., 33 (1996), pp.~627--665.

\bibitem{LY1990}
{\sc R.~J. LeVeque and H.~C. Yee}, {\em A study of numerical methods for
  hyperbolic conservation laws with stiff source terms}, J. Comput. Phys., 86
  (1990), pp.~187--210.

\bibitem{EGM1984}
{\sc P.Embid, J.Goodman, and A.Majda}, {\em Multiple steady states for 1-d
  transonic flow}, SIAM Journal on Scientific and Statistical Computation, 5
  (1984), pp.~21--41.

\bibitem{Spe1987}
{\sc S.~Spekreijse}, {\em Mutligrid solution of monotone second-order
  discretizations of hyperbolic conservation laws}, Math. Comp., 49 (1987),
  pp.~135--155.

\bibitem{S2018}
{\sc S.~Succi}, {\em The lattice boltzmann equation: For complex states of
  flowing matter}.
\newblock Oxford University Press, 2018.

\end{thebibliography}

\end{document}